\documentclass[11pt,reqno]{amsart}
\usepackage[margin=1in]{geometry}
\usepackage{amsmath,amssymb,amsthm,mathtools}
\usepackage{enumitem}
\usepackage[hidelinks]{hyperref}
\hypersetup{pdftitle={Clifford Actions and the Algebra of Double Forms},pdfauthor={A. Belkhirat}}

\newtheorem{theorem}{Theorem}[section]
\newtheorem{proposition}[theorem]{Proposition}
\newtheorem{lemma}[theorem]{Lemma}
\newtheorem{corollary}[theorem]{Corollary}
\theoremstyle{definition}

\theoremstyle{remark}
\newtheorem{remark}[theorem]{Remark}

\newcommand{\D}{\mathcal D}
\newcommand{\Id}{\operatorname{Id}}
\newcommand{\id}{\Id}

\newcommand{\End}{\operatorname{End}}

\newcommand{\Bop}{\mathcal B}
\newcommand{\vol}{\operatorname{vol}}
\newcommand{\ip}[2]{\left\langle #1,#2\right\rangle}
\newcommand{\muL}[1]{{\mu_{#1}^{L}}}
\newcommand{\muR}[1]{{\mu_{#1}^{R}}}
\newcommand{\iL}[1]{{i_{#1}^{L}}}
\newcommand{\iR}[1]{{i_{#1}^{R}}}

\newcommand{\Wedge}{\mathbin{\mathop{\wedge}\nolimits}}

\title{Clifford Actions and the Algebra of Double Forms}
\author{A. Belkhirat}
\address{Department of Mathematics, American International University (AIU), Kuwait}
\email{a.belkhirat@aiu.edu.kw}

\date{}

\begin{document}

\begin{abstract}
Let $\D(V)=\Lambda V^*\otimes\Lambda V^*$ be the algebra of double forms on a Euclidean vector space. We realize its two commuting Clifford actions as regular multiplications in Greub's composition algebra. The generators
\[
 \chi_a=((1\otimes a)-(a\otimes1))e^g
\]
satisfy the Clifford relations and give
\[
 C_R(a)\omega=\chi_a\circ\omega,\qquad
 C_L(a)\omega=-\omega\circ\chi_a.
\]
The coupled action of a $(1,1)$-double form splits into exterior multiplication, double contraction and two generalized Bianchi operators. For the metric these generate two commuting $\mathfrak{sl}_2$-actions. In arbitrary bidegree, contracted wedge products give every homogeneous component. In bidegree $(2,2)$ the central component is $\Gamma_{1,1}(R)=\frac14R\#$, where $\#$ is the sharp product on double forms, defined by Clifford commutators. The extreme components recover the classical curvature decomposition, while same-factor coupling gives the Weitzenb\"ock curvature transformation. We also compute the linearization of the algebraic sharp quadratic map and two geometric examples.
\end{abstract}

\maketitle

\section{Introduction}

Let $V$ be an $n$-dimensional Euclidean vector space.  The algebra of double
forms is
\[
\D(V)=\bigoplus_{p,q=0}^n \D^{p,q}(V),
\qquad
\D^{p,q}(V)=\Lambda^pV^*\otimes\Lambda^qV^*.
\]
Their exterior and composition products were studied by Greub, Vanstone and Greub--Vanstone \cite{Greub,VanstoneNew,VanstonePoincare,VanstoneCanonical,GreubVanstone}, and in double-form notation by Belkhirat--Labbi \cite{BelkhiratLabbi}. We use the same composition algebra to realize the factorwise Clifford actions.

For $h=a\otimes b$, set $\Gamma(h)=C_L(a)C_R(b)$, where $C=\mu-i$. Expansion gives four components with shifts
\[
 (+1,+1),\quad(+1,-1),\quad(-1,+1),\quad(-1,-1),
\]
namely exterior multiplication, two mixed wedge--contraction maps and double contraction. Operator shifts record changes of bidegree, not degrees of double forms.

The metric double form
\[
g=\sum_{i=1}^n e^i\otimes e^i
\]
is distinguished.  In this case the four components become metric
multiplication, the first Bianchi operator, its adjoint, and contraction.  One
of the main elementary observations of the present paper is therefore
\[
\Gamma(g)=\mu_g-\Bop-\Bop^*+c.
\]
A further structural consequence is that $(\mu_g,c)$ and
$(\Bop,\Bop^*)$ generate two commuting $\mathfrak{sl}_2$-actions after adjoining
the appropriate degree operators.

The central result of the paper is the composition-algebra realization of the
Clifford generators.  For $a\in V^*$ we define
\[
\chi_a=((1\otimes a)-(a\otimes1))e^g
\]
and prove
\begin{equation}\label{eq:intro-main-chi}
{
C_R(a)\omega=\chi_a\circ\omega,
\qquad
C_L(a)\omega=-\omega\circ\chi_a.
}
\end{equation}
Consequently, for $h=\sum_\lambda a_\lambda\otimes b_\lambda$,
\[
 \Gamma(h)\omega
 =-\sum_\lambda\chi_{b_\lambda}\circ\omega\circ\chi_{a_\lambda}.
\]
Its four homogeneous pieces are therefore obtained from one two-sided composition action.

The higher Clifford transform gives a second structural layer.  For every
$R\in\D^{2,2}$, without imposing curvature symmetries, $\Gamma(R)$ has nine
homogeneous operator components, distinguished by their bidegree shifts.  Its bidegree-preserving component is
\begin{equation}\label{eq:intro-main-sharp}
 {\Gamma_{1,1}(R)(\omega)=\frac14\,R\#\omega.}
\end{equation}
Here $\#$ is the sharp product on double forms introduced by Labbi \cite{LabbiLaplacians}, defined by Clifford commutators. The identity places this known operation inside the homogeneous transform. For
\[
 \mathcal Q(R)=-\frac14R\#R,
\]
it gives
\[
 \mathcal Q(R)=-\Gamma_{1,1}(R)R,\qquad
 d\mathcal Q_R(S)=-2\Gamma_{1,1}(R)S.
\]
For curvature tensors, the sharp reaction contribution is $-\mathcal Q(R)=\Gamma_{1,1}(R)R$ in the convention fixed below.

Clifford multiplication, the four metric maps and the sharp product are classical \cite{LawsonMichelsohn,LabbiDoubleForms,LabbiBianchi,LabbiLaplacians}. We establish the explicit regular-action dictionary \eqref{eq:intro-main-chi}, its higher homogeneous decomposition and the identification \eqref{eq:intro-main-sharp}. The extreme channels also recover the orthogonal curvature decomposition in Section~\ref{subsec:orthogonal-curvature-channel-recovery}.

Sections~2--5 fix the elementary operators and their symmetries. Sections~6--8 treat the metric commutators and generalized Bianchi maps. Section~9 gives the composition realization, and Section~10 develops the higher channels and curvature applications.

\section{Double forms and elementary operators}

Let $V$ be a real Euclidean vector space of dimension $n$.  We equip each
$\Lambda^pV^*$ with the induced inner product and $\D^{p,q}$ with the product
inner product
\[
\ip{\alpha\otimes\beta}{\gamma\otimes\delta}
=
\ip{\alpha}{\gamma}\,\ip{\beta}{\delta}.
\]
We set $\D^{p,q}=0$ if either index lies outside $\{0,\ldots,n\}$.
Throughout, tensor-product operators are the ordinary, ungraded tensor-product
operators.  Thus an operator acting on the left factor commutes with one acting
on the right factor.

\subsection{Conventions and elementary identities}

We use the musical isomorphisms determined by the Euclidean metric and write
$a^\sharp$ for the vector dual to $a\in V^*$.  The exterior product of double
forms is the factorwise product, denoted by $\cdot$:
\[
(\alpha\otimes\beta)\cdot(\gamma\otimes\delta)
=(\alpha\wedge\gamma)\otimes(\beta\wedge\delta).
\]
No Koszul sign is inserted between the two tensor factors.

\begin{lemma}\label{lem:degree-identities}
Let $(e_i)_{i=1}^n$ be an orthonormal basis of $V$, with dual basis
$(e^i)_{i=1}^n$.  Then, for every $\eta\in\Lambda^pV^*$,
\[
\sum_{i=1}^n e^i\wedge i_{e_i}\eta=p\,\eta,
\qquad
\sum_{i=1}^n i_{e_i}(e^i\wedge\eta)=(n-p)\eta.
\]
\end{lemma}

\begin{proof}
For $\eta=e^{i_1}\wedge\cdots\wedge e^{i_p}$,
\[
 e^i\wedge i_{e_i}\eta=
 \begin{cases}\eta,&i\in\{i_1,\ldots,i_p\},\\0,&\text{otherwise}.
 \end{cases}
\]
Hence $\sum_i e^i\wedge i_{e_i}\eta=p\eta$. The Cartan relation then gives
\[
 \sum_i i_{e_i}(e^i\wedge\eta)
 =n\eta-\sum_i e^i\wedge i_{e_i}\eta=(n-p)\eta.
\]
Extend linearly.
\end{proof}

We reserve $\wedge$ for exterior forms, $\cdot$ for double-form multiplication and $\circ$ for Greub composition. Powers $g^k$ denote exterior powers.

\subsection{Factorwise exterior and interior multiplication}

For $a\in V^*$ define
\[
\muL{a}(\alpha\otimes\beta)
=(a\wedge\alpha)\otimes\beta,
\qquad
\muR{a}(\alpha\otimes\beta)
=\alpha\otimes(a\wedge\beta).
\]
Thus
\[
\muL{a}:\D^{p,q}\to\D^{p+1,q},
\qquad
\muR{a}:\D^{p,q}\to\D^{p,q+1}.
\]
For $v\in V$ define
\[
\iL{v^\flat}(\alpha\otimes\beta)
=(i_v\alpha)\otimes\beta,
\qquad
\iR{v^\flat}(\alpha\otimes\beta)
=\alpha\otimes i_v\beta.
\]
Their bidegree shifts are $(-1,0)$ and $(0,-1)$, respectively.

The ordinary Cartan identity
\[
i_v(a\wedge\alpha)=a(v)\alpha-a\wedge i_v\alpha
\]
gives the canonical anticommutation relations (CAR).  For operators $A,B$,
write $\{A,B\}:=AB+BA$ for the anticommutator.  Then
\begin{align}
\{\iL{v^\flat},\muL{a}\}&=a(v)\id,\label{eq:CARL}\\
\{\iR{v^\flat},\muR{a}\}&=a(v)\id.\label{eq:CARR}
\end{align}
Moreover,
\begin{align}
\muL{a}\muL{b}+\muL{b}\muL{a}&=0,\\
\iL{a}\iL{b}+\iL{b}\iL{a}&=0,
\end{align}
with analogous identities on the right factor.

Since left and right operators act on different tensor factors,
\begin{equation}\label{eq:left-right-commute}
[\muL{a},\muR{b}]
=[\muL{a},\iR{b}]
=[\iL{a},\muR{b}]
=[\iL{a},\iR{b}]=0.
\end{equation}

\section{Two commuting Clifford actions}

Using the metric identification $V^*\simeq V$, define
\begin{equation*}
C_L(a)=\muL{a}-\iL{a},
\qquad
C_R(a)=\muR{a}-\iR{a}.
\end{equation*}

\begin{proposition}\label{prop:Clifford}
For all $a,b\in V^*$,
\begin{align}
C_L(a)C_L(b)+C_L(b)C_L(a)
&=-2\ip{a}{b}\id,\label{eq:cliffL}\\
C_R(a)C_R(b)+C_R(b)C_R(a)
&=-2\ip{a}{b}\id.\label{eq:cliffR}
\end{align}
Moreover,
\begin{equation}\label{eq:LRcliffcommute}
[C_L(a),C_R(b)]=0.
\end{equation}
\end{proposition}

\begin{proof}
Write $\mu_a=\muL{a}$ and $i_a=\iL{a}$. By the CAR,
\[
\begin{aligned}
 \{C_L(a),C_L(b)\}
 &=\{\mu_a,\mu_b\}+\{i_a,i_b\}
   -\{\mu_a,i_b\}-\{i_a,\mu_b\}\\
 &=-2\langle a,b\rangle\id.
\end{aligned}
\]
The right-factor calculation is identical. Equation~\eqref{eq:left-right-commute} gives
\[
 [C_L(a),C_R(b)]
 =[\muL{a}-\iL{a},\muR{b}-\iR{b}]=0.
\]
\end{proof}

Thus $\D(V)$ carries two commuting Clifford representations.

\section{The four-component operator associated with a \texorpdfstring{$(1,1)$}{(1,1)}-double form}

Let first $h=a\otimes b\in\D^{1,1}$.  Define
\[
\Gamma(h):=C_L(a)C_R(b).
\]
The transform $\Gamma$ is algebraic.
Expanding gives
\begin{equation*}
\Gamma(h)
=
\muL{a}\muR{b}
-\muL{a}\iR{b}
-\iL{a}\muR{b}
+\iL{a}\iR{b}.
\end{equation*}

We therefore define
\begin{align}
\mu_h&:=\muL{a}\muR{b},\label{eq:Mh}\\
\Bop_h&:=\muL{a}\iR{b},\label{eq:Bplus}\\
\Bop_h^*&:=\iL{a}\muR{b},\label{eq:Bminus}\\
i_h&:=\iL{a}\iR{b}.\label{eq:Kh}
\end{align}
Their bidegree shifts are respectively
\[
(+1,+1),\qquad(+1,-1),\qquad(-1,+1),\qquad(-1,-1).
\]
For example, a shift $(+1,-1)$ means $\D^{p,q}\to\D^{p+1,q-1}$.

For a general
\[
h=\sum_\lambda a_\lambda\otimes b_\lambda\in\D^{1,1},
\]
extend all four operators and $\Gamma(h)$ linearly in $h$.

\begin{lemma}\label{lem:tensorial-four}
The assignments
\[
(a,b)\longmapsto \muL{a}\muR{b},\quad
\muL{a}\iR{b},\quad
\iL{a}\muR{b},\quad
\iL{a}\iR{b}
\]
are bilinear maps from $V^*\times V^*$ to $\End(\D)$.  Consequently each
induces a unique linear map
\[
\D^{1,1}=V^*\otimes V^*\longrightarrow\End(\D).
\]
In particular $\mu_h,\Bop_h,\Bop_h^*,i_h$ and $\Gamma(h)$ are intrinsic functions of
$h$.
\end{lemma}

\begin{proof}
Exterior multiplication, contraction and the musical map are linear. Each assignment is therefore bilinear and factors uniquely through $V^*\otimes V^*$ by the tensor-product universal property.
\end{proof}

\begin{proposition}\label{prop:fourcomponent}
For every $h\in\D^{1,1}$,
\begin{equation*}
{\Gamma(h)=\mu_h-\Bop_h-\Bop_h^*+i_h.}
\end{equation*}
Moreover,
\begin{equation}\label{eq:Mhexterior}
\mu_h(\omega)=h\cdot\omega,
\end{equation}
where the product on the right is the exterior product of double forms.
\end{proposition}

\begin{proof}
Only \eqref{eq:Mhexterior} requires comment.  For
$h=a\otimes b$ and $\omega=\alpha\otimes\beta$,
\[
\mu_h\omega
=(a\wedge\alpha)\otimes(b\wedge\beta)
=(a\otimes b)(\alpha\otimes\beta).
\]
The general case follows by linearity.
\end{proof}

\section{Adjoints, transpose and Hodge duality}

\subsection{Adjoints}

On ordinary forms exterior multiplication and contraction are adjoint.  Hence
\begin{equation*}
\muL{a}^*=\iL{a},
\qquad
\muR{a}^*=\iR{a}.
\end{equation*}
Consequently
\[
C_L(a)^*=-C_L(a),
\qquad
C_R(a)^*=-C_R(a).
\]

\begin{proposition}\label{prop:adjoints}
For every real $h\in\D^{1,1}$,
\begin{equation*}
\mu_h^*=i_h,
\qquad
(\Bop_h)^*=\Bop_h^*.
\end{equation*}
In particular,
\begin{equation*}
\Gamma(h)^*=\Gamma(h).
\end{equation*}
\end{proposition}

\begin{proof}
For $h=a\otimes b$,
\[
\mu_h^*
=(\muL{a}\muR{b})^*
=\iR{b}\iL{a}
=i_h,
\]
and similarly
\[
(\Bop_h)^*
=\muR{b}\iL{a}
=\Bop_h^*.
\]
Finally, the two Clifford factors commute, and each is skew-adjoint, so
\[
\Gamma(h)^*
=C_R(b)^*C_L(a)^*
=C_R(b)C_L(a)
=C_L(a)C_R(b).
\]
Extend linearly.
\end{proof}

\subsection{Transpose}

For $\omega\in\D^{p,q}$, define its transpose
$\omega^t\in\D^{q,p}$ by
\[
(\alpha\otimes\beta)^t=\beta\otimes\alpha
\]
on simple double forms and extend linearly.  Then
\[
(\omega^t)^t=\omega.
\]
We use the superscript ${}^t$ throughout for this operation.  In particular,
for every $a\in V^*$,
\[
(\muL{a}\omega)^t=\muR{a}(\omega^t),
\qquad
(\iL{a}\omega)^t=\iR{a}(\omega^t),
\]
and similarly with left and right interchanged.  If
\[
h^t=\sum_\lambda b_\lambda\otimes a_\lambda
\qquad\text{for}\qquad
h=\sum_\lambda a_\lambda\otimes b_\lambda,
\]
then
\begin{align}
(\mu_h\omega)^t&=\mu_{h^t}(\omega^t),\\
(i_h\omega)^t&=i_{h^t}(\omega^t),\\
(\Bop_h\omega)^t&=\Bop_{h^t}^*(\omega^t),\\
(\Bop_h^*\omega)^t&=\Bop_{h^t}(\omega^t),
\end{align}
and therefore
\begin{equation*}
(\Gamma(h)\omega)^t=\Gamma(h^t)(\omega^t).
\end{equation*}

\subsection{Hodge-star identities}

Assume now that $V$ is oriented.  Let
\[
*:\Lambda^pV^*\longrightarrow\Lambda^{n-p}V^*
\]
be the Hodge star defined by
\[
\alpha\wedge *\beta=\ip{\alpha}{\beta}\vol.
\]
With this convention,
\[
*^2=(-1)^{p(n-p)}\id
\quad\text{on }\Lambda^pV^*.
\]

For $\eta\in\Lambda^kV^*$ and $a\in V^*$, the two elementary identities are
\begin{align}
*(a\wedge\eta)
&=(-1)^k i_{a^\sharp}(*\eta),\label{eq:hodge-wedge}\\
*(i_{a^\sharp}\eta)
&=(-1)^{k-1}a\wedge(*\eta).\label{eq:hodge-contract}
\end{align}

Define the factorwise stars
\[
*_L=*\otimes\id,
\qquad
*_R=\id\otimes*,
\]
and the double Hodge star
\[
*_D:=*_L*_R=*\otimes*.
\]
For $\omega\in\D^{p,q}$, equations
\eqref{eq:hodge-wedge}--\eqref{eq:hodge-contract} give
\begin{align*}
*_L\muL{a}\omega
&=(-1)^p\iL{a}*_L\omega,\\
*_L\iL{a}\omega
&=(-1)^{p-1}\muL{a}*_L\omega,\\
*_R\muR{b}\omega
&=(-1)^q\iR{b}*_R\omega,\\
*_R\iR{b}\omega
&=(-1)^{q-1}\muR{b}*_R\omega.
\end{align*}

\begin{proposition}\label{prop:hodge-four}
On $\D^{p,q}$ one has
\begin{align}
*_D \mu_h
&=(-1)^{p+q}i_h*_D,\label{eq:H-M}\\
*_D i_h
&=(-1)^{p+q}\mu_h*_D,\label{eq:H-K}\\
*_D \Bop_h
&=(-1)^{p+q-1}\Bop_h^**_D,\label{eq:H-Bp}\\
*_D \Bop_h^*
&=(-1)^{p+q-1}\Bop_h*_D.\label{eq:H-Bm}
\end{align}
\end{proposition}

\begin{proof}
It is enough to treat $h=a\otimes b$, since all identities are linear in
$h$.  Let $\omega\in\D^{p,q}$.  Because $*_L$ commutes with every right
operator and $*_R$ with every left operator,
\[
\begin{aligned}
*_D \mu_h\omega
&=*_L*_R \muL{a}\muR{b}\omega\\
&=*_L\muL{a}\,*_R\muR{b}\omega\\
&=(-1)^p \iL{a}*_L
   \bigl((-1)^q\iR{b}*_R\omega\bigr)\\
&=(-1)^{p+q}\iL{a}\iR{b}*_D\omega,
\end{aligned}
\]
which is \eqref{eq:H-M}.  For the double contraction one similarly obtains
\[
\begin{aligned}
*_D i_h\omega
&=*_L\iL{a}\,*_R\iR{b}\omega\\
&=(-1)^{p-1}\muL{a}*_L
   \bigl((-1)^{q-1}\muR{b}*_R\omega\bigr)\\
&=(-1)^{p+q}\mu_h*_D\omega.
\end{aligned}
\]
For the first mixed operator,
\[
\begin{aligned}
*_D \Bop_h\omega
&=*_L\muL{a}\,*_R\iR{b}\omega\\
&=(-1)^p\iL{a}*_L
   \bigl((-1)^{q-1}\muR{b}*_R\omega\bigr)\\
&=(-1)^{p+q-1}\Bop_h^**_D\omega,
\end{aligned}
\]
and the last identity follows in exactly the same manner:
\[
*_D \Bop_h^*\omega
=(-1)^{p+q-1}\Bop_h*_D\omega.
\]
No additional graded sign occurs because the left and right tensor-factor
operators commute in our convention.
\end{proof}

\section{The metric double form}

Let $(e_1,\dots,e_n)$ be an orthonormal basis, with dual basis
$(e^1,\dots,e^n)$, and write
\[
g=\sum_{i=1}^n e^i\otimes e^i\in\D^{1,1}.
\]
Then
\[
\Gamma(g)=\sum_{i=1}^n C_L(e^i)C_R(e^i).
\]

Define
\begin{align}
\mu_g&:=\sum_i \muL{e^i}\muR{e^i},\label{eq:Lg}\\
c&:=\sum_i \iL{e^i}\iR{e^i},\label{eq:c}\\
\Bop&:=\sum_i \muL{e^i}\iR{e^i},\label{eq:S}\\
\Bop^*&:=\sum_i \iL{e^i}\muR{e^i}.\label{eq:Sadj}
\end{align}
Thus $\mu_g$ is exterior multiplication by the metric and $c$ is the usual
contraction operator.  The operator \eqref{eq:S} is the classical first
Bianchi map in the standard unnormalized double-form convention; its adjoint is
\eqref{eq:Sadj}.

\begin{theorem}\label{thm:metric-clifford}
The four homogeneous components of $\Gamma(g)$ are
\[
\mu_g,\qquad -\Bop,\qquad -\Bop^*,\qquad c,
\]
and
\begin{equation}\label{eq:metric-clifford}
{
\Gamma(g)=\mu_g-\Bop-\Bop^*+c.
}
\end{equation}
\end{theorem}

\begin{proof}
This is Proposition \ref{prop:fourcomponent} applied to
$g=\sum_i e^i\otimes e^i$ together with the definitions
\eqref{eq:Lg}--\eqref{eq:Sadj}.
\end{proof}

\begin{remark}\label{rem:labbi-four-maps}
In Labbi's notation \cite[Propositions~3.1--3.2]{LabbiBianchi},
\[
 L_g^{1,1}=\mu_g,\quad L_g^{1,-1}=\Bop,\quad
 L_g^{-1,1}=\Bop^*,\quad L_g^{-1,-1}=c.
\]
Thus $\Bop=S$ and $\Bop^*=\widetilde S$ in that reference. These classical maps give the homogeneous pieces $\mu_g,-\Bop,-\Bop^*,c$ of $\Gamma(g)$.
\end{remark}

\begin{remark}\label{rem:basis-independent}
Although \eqref{eq:metric-clifford} is written in an orthonormal basis, every
term is intrinsic.  In particular, $\Gamma(g)$ is the contraction of the left
and right Clifford actions against the metric.
\end{remark}

\begin{corollary}\label{cor:curvature-simple}
Let $R\in\D^{2,2}$ satisfy $R^t=R$ and the first Bianchi identity
$\Bop R=0$.  Then
\[
\Bop^*R=0
\]
and hence
\begin{equation}\label{eq:gamma-curvature}
\Gamma(g)R=\mu_gR+cR.
\end{equation}
\end{corollary}

\begin{proof}
Since $g^t=g$, the transpose relations of Section 5 give
\[
(\Bop\omega)^t=\Bop^*(\omega^t).
\]
Because $R^t=R$ and $\Bop R=0$, it follows that
\[
\Bop^*R=(\Bop R)^t=0.
\]
Equation \eqref{eq:gamma-curvature} follows from
\eqref{eq:metric-clifford}.
\end{proof}

\section{Commutator algebra in the metric case}\label{sec:commutators}

The metric specialization admits a stronger algebraic organization.  We now
compute all commutators among
\[
\mu_g,\qquad c,\qquad \Bop,\qquad \Bop^*.
\]

\subsection{Degree operators}

Define
\begin{equation*}
N_L:=\sum_{i=1}^n \muL{e^i}\iL{e^i},
\qquad
N_R:=\sum_{i=1}^n \muR{e^i}\iR{e^i}.
\end{equation*}
These are independent of the chosen orthonormal basis.  On $\D^{p,q}$,
\begin{equation}\label{eq:degree-eigen}
N_L=p\,\id,
\qquad
N_R=q\,\id.
\end{equation}
Set
\begin{equation}\label{eq:Hpm}
H_+:=N_L+N_R-n\,\id,
\qquad
H_-:=N_L-N_R.
\end{equation}
Thus on $\D^{p,q}$,
\begin{equation*}
H_+=(p+q-n)\id,
\qquad
H_-=(p-q)\id.
\end{equation*}

\subsection{Metric multiplication and contraction}

Kulkarni's metric--contraction identity \cite{KulkarniBianchi}, developed further in \cite{LabbiDoubleForms}, is
\begin{equation}\label{eq:Kulkarni-metric-contraction}
 c(g\cdot\omega)=g\cdot(c\omega)+(n-p-q)\omega,
 \qquad\omega\in\D^{p,q}.
\end{equation}

\begin{proposition}\label{prop:Lc-sl2}
The operators $\mu_g,c,H_+$ satisfy
\begin{align}
[\mu_g,c]&=H_+,\label{eq:Lc}\\
[H_+,\mu_g]&=2\mu_g,\label{eq:HplusL}\\
[H_+,c]&=-2c.\label{eq:Hplusc}
\end{align}
Equivalently, on $\D^{p,q}$,
\begin{equation*}
c\mu_g-\mu_gc=(n-p-q)\id.
\end{equation*}
\end{proposition}

\begin{proof}
For $\omega\in\D^{p,q}$, \eqref{eq:Kulkarni-metric-contraction} gives
\[
 [\mu_g,c]\omega=(p+q-n)\omega=H_+\omega.
\]
Since $\mu_g$ raises both degrees and $c$ lowers both,
\[
\begin{aligned}
 [H_+,\mu_g]\omega
 &=\bigl((p+q+2-n)-(p+q-n)\bigr)\mu_g\omega
 =2\mu_g\omega,\\
 [H_+,c]\omega
 &=\bigl((p+q-2-n)-(p+q-n)\bigr)c\omega=-2c\omega.
\end{aligned}
\]
\end{proof}

Thus $(\mu_g,c,H_+)$ is an $\mathfrak{sl}_2$-triple with the convention
\[
[E,F]=H,
\qquad
[H,E]=2E,
\qquad
[H,F]=-2F.
\]

\subsection{The Bianchi pair}

\begin{proposition}\label{prop:S-sl2}
The operators $\Bop,\Bop^*,H_-$ satisfy
\begin{align}
[\Bop,\Bop^*]&=H_-,\label{eq:SSstar}\\
[H_-,\Bop]&=2\Bop,\label{eq:HminusS}\\
[H_-,\Bop^*]&=-2\Bop^*.\label{eq:HminusSstar}
\end{align}
In particular, on $\D^{p,q}$,
\begin{equation*}
(\Bop\Bop^*-\Bop^*\Bop)|_{\D^{p,q}}
=(p-q)\id.
\end{equation*}
\end{proposition}

\begin{proof}
We calculate
\[
\begin{aligned}
\Bop^*\Bop
&=\sum_{i,j}\iL{e^i}\muR{e^i}\muL{e^j}\iR{e^j}\\
&=\sum_{i,j}\iL{e^i}\muL{e^j}\muR{e^i}\iR{e^j}\\
&=N_R-
\sum_{i,j}\muL{e^j}\iL{e^i}\muR{e^i}\iR{e^j},
\end{aligned}
\]
whereas
\[
\begin{aligned}
\Bop\Bop^*
&=\sum_{i,j}\muL{e^j}\iR{e^j}\iL{e^i}\muR{e^i}\\
&=N_L-
\sum_{i,j}\muL{e^j}\iL{e^i}\muR{e^i}\iR{e^j}.
\end{aligned}
\]
Subtracting gives
\[
[\Bop,\Bop^*]=N_L-N_R=H_-.
\]
Since $\Bop$ has bidegree shift $(+1,-1)$ and $\Bop^*$ has bidegree shift $(-1,+1)$,
\eqref{eq:HminusS} and \eqref{eq:HminusSstar} follow.
\end{proof}

Thus $(\Bop,\Bop^*,H_-)$ is a second $\mathfrak{sl}_2$-triple.

\subsection{The two triples commute}

\begin{proposition}\label{prop:two-sl2-commute}
The two $\mathfrak{sl}_2$-triples commute with one another.  More precisely,
\begin{equation*}
[\mu_g,\Bop]=[\mu_g,\Bop^*]=[c,\Bop]=[c,\Bop^*]=0,
\end{equation*}
and
\begin{equation}\label{eq:Hcross}
[H_+,\Bop]=[H_+,\Bop^*]=[H_-,\mu_g]=[H_-,c]=0.
\end{equation}
\end{proposition}

\begin{proof}
We prove the first identity; the remaining three follow by the same calculation
or by adjunction and transpose.  Using the commutation of left and right
operators,
\[
\begin{aligned}
\mu_g\Bop
&=\sum_{i,j}\muL{e^i}\muL{e^j}\muR{e^i}\iR{e^j},\\
\Bop \mu_g
&=\sum_{i,j}\muL{e^j}\muL{e^i}\iR{e^j}\muR{e^i}.
\end{aligned}
\]
Now
\[
\muL{e^j}\muL{e^i}=-\muL{e^i}\muL{e^j}
\]
and
\[
\iR{e^j}\muR{e^i}=\delta_{ij}\id-\muR{e^i}\iR{e^j}.
\]
The Kronecker term vanishes because $\muL{e^i}^2=0$, leaving precisely
$\mu_g\Bop$.  Hence $[\mu_g,\Bop]=0$.

The identities in \eqref{eq:Hcross} follow immediately from bidegrees:
$\Bop$ and $\Bop^*$ preserve $p+q$, while $\mu_g$ and $c$ preserve $p-q$.
\end{proof}

\begin{theorem}\label{thm:sl2plus}
On the algebra of double forms there are two commuting representations of
$\mathfrak{sl}_2$ generated by
\[
(\mu_g,c,H_+)
\qquad\text{and}\qquad
(\Bop,\Bop^*,H_-),
\]
where
\[
H_+=N_L+N_R-n\id,
\qquad
H_-=N_L-N_R.
\]
Equivalently, these operators define a representation of
\[
\mathfrak{sl}_2\oplus\mathfrak{sl}_2
\]
on $\D(V)$.
\end{theorem}

\begin{proof}
Propositions \ref{prop:Lc-sl2}, \ref{prop:S-sl2}, and
\ref{prop:two-sl2-commute} give the defining relations.
\end{proof}

\subsection{Primitive decomposition and orthogonality}

For $0\le p,q\le n$, define
\[
P^{p,q}:=\ker\bigl(c:\D^{p,q}\longrightarrow\D^{p-1,q-1}\bigr),
\]
where the target is understood to be zero when $p=0$ or $q=0$.  Thus a
double form is primitive precisely when it is trace-free with respect to the
metric contraction.

\begin{proposition}\label{prop:primitive-metric-chains}
Let $0\ne\omega\in P^{a,b}$ and put
\[
m=n-a-b.
\]
Then $m\ge0$, and
\begin{equation}\label{eq:primitive-metric-chain}
c\bigl(g^j\cdot\omega\bigr)
=j(m-j+1)g^{j-1}\cdot\omega,
\qquad 1\le j\le m,
\end{equation}
while
\[
g^{m+1}\cdot\omega=0.
\]
Consequently
\[
\omega,\quad g\cdot\omega,\quad\ldots,\quad g^m\cdot\omega
\]
spans an irreducible $(m+1)$-dimensional module for the metric
$\mathfrak{sl}_2$-action.
\end{proposition}

\begin{proof}
Put $L=\mu_g$. For $c\omega=0$ and $m=n-a-b$,
\[
 [c,L]|_{\D^{a+j,b+j}}=(m-2j)\Id.
\]
Consequently, for every $j\ge1$,
\[
\begin{aligned}
 cL^j\omega
 &=[c,L^j]\omega
 =\sum_{r=0}^{j-1}L^r[c,L]L^{j-1-r}\omega\\
 &=\sum_{r=0}^{j-1}\bigl(m-2(j-1-r)\bigr)L^{j-1}\omega\\
 &=j(m-j+1)L^{j-1}\omega.
\end{aligned}
\]
Since $L^*=c$,
\[
 \|L^j\omega\|^2
 =j(m-j+1)\|L^{j-1}\omega\|^2.
\]
For $j=1$ this gives $m\ge0$; for $1\le j\le m$ the norms are positive, and for $j=m+1$ they vanish. Thus the displayed chain is nonzero up to $L^m\omega$, terminates at $L^{m+1}\omega=0$, and has the standard irreducible $\mathfrak{sl}_2$ action.
\end{proof}

\begin{theorem}\label{thm:primitive-orthogonal-decomposition}
For every bidegree $(p,q)$ one has the canonical decomposition
\begin{equation}\label{eq:primitive-orthogonal-decomposition}
{
\D^{p,q}
=
\bigoplus_{k=\max\{0,p+q-n\}}^{\min\{p,q\}}
 g^k\cdot P^{p-k,q-k}.}
\end{equation}
Moreover, the summands in \eqref{eq:primitive-orthogonal-decomposition} are
mutually orthogonal with respect to the Euclidean inner product induced by
$g$ on $\D^{p,q}$.
\end{theorem}

\begin{proof}
The operator $H_-=N_L-N_R$ commutes with the metric $\mathfrak{sl}_2$.
Apply complete reducibility on each $H_-$-eigenspace. Joint $H_+$- and $H_-$-weights determine bidegree, so its lowest-weight vectors are homogeneous and lie in $\ker c$.
Proposition~\ref{prop:primitive-metric-chains} gives the direct sum decomposition.  A chain starting at $(p-k,q-k)$ reaches $(p,q)$ precisely when
\[
 0\le k\le\min\{p,q\},\qquad
 k\le n-(p-k)-(q-k)
 \quad\Longleftrightarrow\quad k\ge p+q-n.
\]

For orthogonality use $\mu_g^*=c$:
\begin{equation}\label{eq:L-c-adjoint-primitive}
\langle g\cdot\alpha,\beta\rangle
=\langle\alpha,c\beta\rangle.
\end{equation}
Let
\[
\alpha\in P^{p-k,q-k},\qquad
\beta\in P^{p-\ell,q-\ell},
\qquad k>\ell.
\]
Repeated use of \eqref{eq:L-c-adjoint-primitive} and
\eqref{eq:primitive-metric-chain} gives a nonzero scalar $C$ (whenever the
corresponding chain terms are nonzero) such that
\[
\begin{aligned}
\langle g^k\cdot\alpha,g^\ell\cdot\beta\rangle
&=\langle c^\ell(g^k\cdot\alpha),\beta\rangle\\
&=C\,\langle g^{k-\ell}\cdot\alpha,\beta\rangle\\
&=C\,\langle\alpha,c^{k-\ell}\beta\rangle=0,
\end{aligned}
\]
because $k-\ell\ge1$ and $c\beta=0$.  Hence distinct primitive metric
levels are orthogonal.
\end{proof}

\begin{remark}
The decomposition above is the double-form analogue of a Lefschetz
primitive decomposition.  Its algebraic origin is the metric--contraction
$\mathfrak{sl}_2$-action, while its orthogonality is a consequence of the
adjoint relation $\mu_g^*=c$.  In particular, the primitive component is
characterized intrinsically as the orthogonal complement
\[
P^{p,q}
=\bigl(g\cdot\D^{p-1,q-1}\bigr)^\perp.
\]
The identification of primitive double forms with $\ker c$ and the
metric--contraction identities are classical in double-form calculus
\cite{KulkarniBianchi,LabbiDoubleForms}.  The highest-weight chain formula
may also be read from the standard $\mathfrak{sl}_2$ calculation in
Fulton--Harris, Lecture~11, \S 11.1, Claim~11.4 and formula~(11.5)
\cite{FultonHarris}.
\end{remark}

\begin{remark}\label{rem:gamma-sl2}
The metric Clifford operator can therefore be written as
\[
\Gamma(g)=(\mu_g+c)-(\Bop+\Bop^*),
\]
where the two parenthesized terms belong to the two commuting
$\mathfrak{sl}_2$-actions above.  This places the four classical double-form
operators appearing in Theorem \ref{thm:metric-clifford} inside a single
rank-two Lie-algebraic structure.
\end{remark}

\begin{remark}\label{rem:sl2-historical}
Both triples consist of classical double-form maps \cite{KulkarniBianchi,LabbiDoubleForms,LabbiBianchi}. Their commutators follow from the CAR; the coupled Clifford operator organizes them as a commuting rank-two action.
\end{remark}

\section{Generalized Bianchi operators}

For $h=\sum_\lambda a_\lambda\otimes b_\lambda\in\D^{1,1}$ we use the notation
\[
\Bop_h:=\sum_\lambda \muL{a_\lambda}\iR{b_\lambda},
\qquad
\Bop_h^*:=\sum_\lambda \iL{a_\lambda}\muR{b_\lambda}.
\]
Thus $\Bop_g=\Bop$ and
\begin{equation}\label{eq:general-h-final}
\Gamma(h)=\mu_h-\Bop_h-\Bop_h^*+i_h,
\end{equation}
where $\mu_h$ denotes exterior multiplication by $h$ and $i_h=\mu_h^*$.
This notation is deliberately close to the one used in \cite{BelkhiratLabbi}.
The four bidegree shifts are respectively $(+1,+1)$, $(+1,-1)$,
$(-1,+1)$ and $(-1,-1)$.

\section{Historical position and the Greub composition algebra}\label{sec:greub}
\label{sec:historical-composition}

We use Greub's mixed exterior algebra \cite[Chapters~6--7]{Greub} and its double-form formulation \cite{BelkhiratLabbi}. Vanstone's Poincar\'e map and decompositions, and the Greub--Vanstone identity, describe the interaction of its two products \cite{VanstonePoincare,VanstoneCanonical,GreubVanstone}. The curvature calculus is developed in \cite{LabbiDoubleForms,LabbiAlgebraic,LabbiBianchi,LabbiWeitzenbock,LabbiLaplacians}. Here we realize the two Clifford actions by explicit inhomogeneous elements of this composition algebra.

Recall the canonical map
\[
\mathcal T:\D(V)\longrightarrow\End(\Lambda V),
\qquad
\mathcal T(\alpha\otimes\beta)(\theta)
=\ip{\alpha^\sharp}{\theta}\,\beta^\sharp,
\]
with the evident degree convention.  Pulling endomorphism composition back by
$\mathcal T$ gives Greub's product
\begin{equation*}
(\alpha\otimes\beta)\circ(\gamma\otimes\delta)
=\ip{\alpha}{\delta}\,\gamma\otimes\beta,
\end{equation*}
whenever the degrees are compatible, and zero otherwise
\cite{BelkhiratLabbi}.  Its unit is the inhomogeneous double form
\begin{equation}\label{eq:composition-unit}
\mathbf 1_\circ=e^g:=\sum_{r=0}^n\frac{g^r}{r!}.
\end{equation}

\subsection{Exterior and interior multiplication inside \texorpdfstring{$(\D,\circ)$}{the composition algebra}}

For $a\in V^*$ define
\begin{equation}\label{eq:eps-iota-forms}
\varepsilon_a:=(1\otimes a)e^g,
\qquad
\iota_a:=(a\otimes1)e^g,
\end{equation}
where the products on the right-hand side are exterior products.

\begin{remark}
Distinguish the operator $\epsilon(a)\theta=a^\sharp\wedge\theta$ on $\Lambda V$, the factorwise operator $\mu_a^R$ on $\D$, and the element $\varepsilon_a\in\D$. They satisfy
\[
 \mathcal T(\varepsilon_a)=\epsilon(a),\qquad
 \varepsilon_a\circ\omega=\mu_a^R\omega.
\]
The exterior product $(1\otimes a)e^g$ constructs the element; composition makes it act on every degree.
\end{remark}

\begin{proposition}\label{prop:T-eps-iota}
Under $\mathcal T$ one has
\[
\mathcal T(\varepsilon_a)=\epsilon(a),
\qquad
\mathcal T(\iota_a)=i_{a^\sharp},
\]
where $\epsilon(a)\theta=a^\sharp\wedge\theta$.
Consequently
\begin{align}
\varepsilon_a\circ\varepsilon_b+\varepsilon_b\circ\varepsilon_a&=0,\label{eq:CAR1}\\
\iota_a\circ\iota_b+\iota_b\circ\iota_a&=0,\label{eq:CAR2}\\
\iota_a\circ\varepsilon_b+\varepsilon_b\circ\iota_a
&=\ip{a}{b}e^g.\label{eq:CAR3}
\end{align}
\end{proposition}

\begin{proof}
Choose an orthonormal basis.  Since
\[
\frac{g^r}{r!}=\sum_{|I|=r}e^I\otimes e^I,
\]
the component of $\varepsilon_a$ acting on $\Lambda^rV$ is
$\sum_{|I|=r}e^I\otimes(a\wedge e^I)$.  Applying $\mathcal T$ gives exterior
multiplication by $a$.  Similarly,
$\iota_a=\sum_I(a\wedge e^I)\otimes e^I$, and the adjoint identity
$\ip{a\wedge e^I}{e^J}=\ip{e^I}{i_{a^\sharp}e^J}$ gives contraction.
Equations \eqref{eq:CAR1}--\eqref{eq:CAR3} are therefore the usual
canonical anticommutation relations transported through $\mathcal T$; the scalar
identity operator pulls back to $e^g$ by \eqref{eq:composition-unit}.
\end{proof}

Define the inhomogeneous double form
\begin{equation}\label{eq:chi-a}
\chi_a:=\varepsilon_a-\iota_a
=\bigl((1\otimes a)-(a\otimes1)\bigr)e^g.
\end{equation}

\begin{corollary}
\label{cor:clifford-composition}
For $a,b\in V^*$,
\begin{equation}\label{eq:chi-clifford}
\chi_a\circ\chi_b+\chi_b\circ\chi_a
=-2\ip{a}{b}e^g.
\end{equation}
Thus $a\mapsto\chi_a$ extends, by the universal property, to the usual
Clifford representation realized inside $(\D,\circ)$, with unit $e^g$.
\end{corollary}

\subsection{The left--right Clifford/Greub dictionary}

\begin{proposition}\label{prop:regular-actions}
For every $a\in V^*$ and $\omega\in\D$,
\begin{align}
\varepsilon_a\circ\omega&=\muR{a}\omega,&
\iota_a\circ\omega&=\iR{a}\omega,\label{eq:left-reg}\\
\omega\circ\varepsilon_a&=\iL{a}\omega,&
\omega\circ\iota_a&=\muL{a}\omega.\label{eq:right-reg}
\end{align}
Hence
\begin{equation*}
C_R(a)\omega=\chi_a\circ\omega,
\qquad
C_L(a)\omega=-\omega\circ\chi_a.
\end{equation*}
\end{proposition}

\begin{proof}
For $\omega=\alpha\otimes\beta$, write $\mathcal T(\omega)=\beta^\sharp\otimes(\alpha^\sharp)^*$. For every endomorphism $A$,
\[
 A\mathcal T(\omega)=(A\beta^\sharp)\otimes(\alpha^\sharp)^*,\qquad
 \mathcal T(\omega)A=\beta^\sharp\otimes(A^*\alpha^\sharp)^*.
\]
Apply these formulas to $A=\epsilon(a)$ and $A=i_{a^\sharp}$, using
\[
 \mathcal T(\varepsilon_a)=\epsilon(a),\quad
 \mathcal T(\iota_a)=i_{a^\sharp},\quad
 \epsilon(a)^*=i_{a^\sharp}.
\]
Injectivity of $\mathcal T$ gives \eqref{eq:left-reg}--\eqref{eq:right-reg}. Subtraction yields
\[
 (\varepsilon_a-\iota_a)\circ\omega=C_R(a)\omega,\qquad
 \omega\circ(\varepsilon_a-\iota_a)=-C_L(a)\omega.
\]
\end{proof}

The commutation $[C_L(a),C_R(b)]=0$ is now immediate from associativity of
$\circ$: left and right regular multiplications commute.

\begin{theorem}
\label{thm:clifford-greub}
Let $h=\sum_\lambda a_\lambda\otimes b_\lambda\in\D^{1,1}$.  Then, for every
$\omega\in\D$,
\begin{equation}\label{eq:Gamma-composition}
{\quad
\Gamma(h)\omega
=-\sum_\lambda\chi_{b_\lambda}\circ\omega\circ\chi_{a_\lambda}.
\quad}
\end{equation}
The four basic operators have the composition realizations
\begin{align*}
\mu_h(\omega)
&=\sum_\lambda\varepsilon_{b_\lambda}\circ\omega\circ\iota_{a_\lambda},\\
\Bop_h(\omega)
&=\sum_\lambda\iota_{b_\lambda}\circ\omega\circ\iota_{a_\lambda},\\
\Bop_h^*(\omega)
&=\sum_\lambda\varepsilon_{b_\lambda}\circ\omega\circ\varepsilon_{a_\lambda},\\
i_h(\omega)
&=\sum_\lambda\iota_{b_\lambda}\circ\omega\circ\varepsilon_{a_\lambda}.
\end{align*}
\end{theorem}

\begin{proof}
For $h=a\otimes b$, equations \eqref{eq:left-reg}--\eqref{eq:right-reg} give
\[
\mu_h\omega=\varepsilon_b\circ\omega\circ\iota_a,
\quad
\Bop_h\omega=\iota_b\circ\omega\circ\iota_a,
\]
\[
\Bop_h^*\omega=\varepsilon_b\circ\omega\circ\varepsilon_a,
\quad
i_h\omega=\iota_b\circ\omega\circ\varepsilon_a.
\]
Substitution into \eqref{eq:general-h-final} and factorization yield
$\Gamma(a\otimes b)\omega=-\chi_b\circ\omega\circ\chi_a$.
Linearity in $h$ proves the theorem.
\end{proof}

For the metric this specializes to
\begin{equation*}
g\omega-\Bop\omega-\Bop^*\omega+c\omega
=-\sum_i\chi_{e^i}\circ\omega\circ\chi_{e^i}.
\end{equation*}
This is the composition-algebra form of the metric Clifford identity.

\subsection{Transpose compatibility in the composition algebra}

Transposition reverses Greub composition:
\begin{equation}\label{eq:transpose-composition}
(\omega_1\circ\omega_2)^t
=\omega_2^t\circ\omega_1^t.
\end{equation}
Moreover, from \eqref{eq:eps-iota-forms} and $g^t=g$,
\[
\varepsilon_a^t=\iota_a,
\qquad
\iota_a^t=\varepsilon_a,
\]
and therefore
\begin{equation}\label{eq:chi-transpose}
{\chi_a^t=-\chi_a.}
\end{equation}

\begin{proposition}
\label{prop:composition-transpose}
For $h\in\D^{1,1}$ and $\omega\in\D$,
\[
\bigl(\Gamma(h)\omega\bigr)^t
=\Gamma(h^t)(\omega^t).
\]
Equivalently, transposition exchanges the two tensor factors in both
$h$ and the argument of $\Gamma(h)$.
\end{proposition}

\begin{proof}
By Theorem~\ref{thm:clifford-greub},
\[
\Gamma(h)\omega
=-\sum_\lambda
\chi_{b_\lambda}\circ\omega\circ\chi_{a_\lambda}.
\]
Using \eqref{eq:transpose-composition} twice and
\eqref{eq:chi-transpose},
\[
\begin{aligned}
(\Gamma(h)\omega)^t
&=-\sum_\lambda
\chi_{a_\lambda}^t\circ\omega^t\circ\chi_{b_\lambda}^t\\
&=-\sum_\lambda
\chi_{a_\lambda}\circ\omega^t\circ\chi_{b_\lambda}.
\end{aligned}
\]
But
$h^t=\sum_\lambda b_\lambda\otimes a_\lambda$, so the last expression is
exactly $\Gamma(h^t)(\omega^t)$.
\end{proof}

\begin{remark}
In \eqref{eq:Gamma-composition}, transposition reverses the order of composition and contributes two cancelling signs from $\chi_a^t=-\chi_a$. This agrees with Proposition~\ref{prop:adjoints}.
\end{remark}

The relation with \cite{BelkhiratLabbi} is particularly transparent.  There,
for a bilinear form $h$, the exterior extensions satisfy
\[
\widehat h_R(\omega)=e^h\circ\omega,
\qquad
\widehat h_L(\omega)=\omega\circ e^{h^t}.
\]
Here the parallel identities are
\[
C_R(a)\omega=\chi_a\circ\omega,
\qquad
C_L(a)\omega=-\omega\circ\chi_a.
\]
Thus the passage from exterior extensions to Clifford actions corresponds,
inside the same Greub composition algebra, to replacing the exponential
$e^h$ by the Clifford generator $\chi_a$.

\section{Arbitrary bidegree: the higher left--right Clifford transform}
\label{sec:higher-bidegree}

For arbitrary bidegree we use the alternating inclusion of exterior powers into the Clifford algebra.

For a multilinear expression $F(a_1,\ldots,a_p)$, write
\[
\operatorname{Alt}F(a_1,\ldots,a_p)
:=\frac1{p!}\sum_{\sigma\in S_p}\operatorname{sgn}(\sigma)
F(a_{\sigma(1)},\ldots,a_{\sigma(p)}).
\]
Thus $\operatorname{Alt}$ denotes antisymmetrization in the displayed
one-form arguments; it should not be confused with an operation on an already
alternating $p$-form.  For
$\alpha=a_1\wedge\cdots\wedge a_p\in\Lambda^pV^*$ define
\begin{equation}\label{eq:chi-alpha}
{\chi(\alpha)
:=\operatorname{Alt}\!\left(
\chi_{a_1}\circ\cdots\circ\chi_{a_p}
\right)
=\frac1{p!}\sum_{\sigma\in S_p}\operatorname{sgn}(\sigma)
\chi_{a_{\sigma(1)}}\circ\cdots\circ\chi_{a_{\sigma(p)}}.}
\end{equation}
For $p=0$ we set $\chi(1)=e^g$.
The displayed map is alternating and therefore extends linearly to all of $\Lambda^pV^*$.
By \eqref{eq:chi-clifford} and the universal property, the map
$a\mapsto\chi_a$ extends to an algebra homomorphism
\[
\mathrm{Cl}(V^*,-g)\longrightarrow(\D,\circ),
\]
whose unit is $e^g$.  We denote by $\operatorname{rev}$ the standard
\emph{Clifford reversion}, characterized by reversal of Clifford products,
\[
\operatorname{rev}(x_1\circ\cdots\circ x_p)
=x_p\circ\cdots\circ x_1.
\]
We reserve $\operatorname{rev}$ for Clifford reversion and ${}^t$ for double-form transpose. On the image of a homogeneous
$p$-form under \eqref{eq:chi-alpha},
\begin{equation}\label{eq:reversion-sign}
\operatorname{rev}(\chi(\alpha))
=(-1)^{p(p-1)/2}\chi(\alpha).
\end{equation}

Likewise define the alternating higher Clifford actions by the same
antisymmetrization convention:
\begin{align*}
C_L(\alpha)
&:=\operatorname{Alt}\!\left(
C_L(a_1)\cdots C_L(a_p)
\right)
=\frac1{p!}\sum_{\sigma\in S_p}\operatorname{sgn}(\sigma)
C_L(a_{\sigma(1)})\cdots C_L(a_{\sigma(p)}),\\
C_R(\beta)
&:=\operatorname{Alt}\!\left(
C_R(b_1)\cdots C_R(b_q)
\right)
=\frac1{q!}\sum_{\tau\in S_q}\operatorname{sgn}(\tau)
C_R(b_{\tau(1)})\cdots C_R(b_{\tau(q)}).
\end{align*}
For degree zero, $C_L(1)=C_R(1)=\Id$.
These definitions are independent of the decompositions of $\alpha$ and
$\beta$.

\begin{proposition}\label{prop:higher-regular}
For $\alpha\in\Lambda^pV^*$, $\beta\in\Lambda^qV^*$ and $\omega\in\D$,
\begin{align}
C_R(\beta)\omega&=\chi(\beta)\circ\omega,\label{eq:CR-higher}\\
C_L(\alpha)\omega&=(-1)^p\omega\circ
\operatorname{rev}(\chi(\alpha)).\label{eq:CL-higher}
\end{align}
Equivalently,
\[
C_L(\alpha)\omega
=(-1)^{p(p+1)/2}\omega\circ\chi(\alpha).
\]
\end{proposition}

\begin{proof}
For $L_x(\omega)=x\circ\omega$ and $R_x(\omega)=\omega\circ x$, associativity gives
\[
 L_xL_y=L_{x\circ y},\qquad R_xR_y=R_{y\circ x}.
\]
Proposition~\ref{prop:regular-actions} therefore yields
\[
\begin{aligned}
 C_R(b_1)\cdots C_R(b_q)\omega
 &=\chi_{b_1}\circ\cdots\circ\chi_{b_q}\circ\omega,\\
 C_L(a_1)\cdots C_L(a_p)\omega
 &=(-1)^p\omega\circ\chi_{a_p}\circ\cdots\circ\chi_{a_1}.
\end{aligned}
\]
Antisymmetrization proves \eqref{eq:CR-higher}--\eqref{eq:CL-higher}. Finally,
\[
 (-1)^p\operatorname{rev}(\chi(\alpha))
 =(-1)^{p+p(p-1)/2}\chi(\alpha)
 =(-1)^{p(p+1)/2}\chi(\alpha).
\]
\end{proof}

For a simple double form $\eta=\alpha\otimes\beta\in\D^{p,q}$ define
\[
\Gamma(\eta):=C_L(\alpha)C_R(\beta),
\]
and extend linearly to arbitrary $\eta\in\D^{p,q}$.

\begin{theorem}
\label{thm:higher-transform}
If
\[
\eta=\sum_\lambda\alpha_\lambda\otimes\beta_\lambda\in\D^{p,q},
\]
then for every $\omega\in\D$,
\begin{equation}\label{eq:higher-transform}
{\quad
\Gamma(\eta)\omega
=(-1)^p\sum_\lambda
\chi(\beta_\lambda)\circ\omega\circ
\operatorname{rev}(\chi(\alpha_\lambda)).
\quad}
\end{equation}
Equivalently,
\begin{equation}\label{eq:higher-transform-sign}
\Gamma(\eta)\omega
=(-1)^{p(p+1)/2}\sum_\lambda
\chi(\beta_\lambda)\circ\omega\circ\chi(\alpha_\lambda).
\end{equation}
For $(p,q)=(1,1)$ this reduces exactly to
Theorem~\ref{thm:clifford-greub}.
\end{theorem}

\begin{proof}
For $\eta=\alpha\otimes\beta$, Proposition~\ref{prop:higher-regular} gives
\[
\begin{aligned}
 \Gamma(\eta)\omega
 &=C_L(\alpha)C_R(\beta)\omega\\
 &=(-1)^p\chi(\beta)\circ\omega\circ\operatorname{rev}(\chi(\alpha))\\
 &=(-1)^{p(p+1)/2}\chi(\beta)\circ\omega\circ\chi(\alpha).
\end{aligned}
\]
Extend linearly. For $p=q=1$ this is \eqref{eq:Gamma-composition}.
\end{proof}

\subsection{Homogeneous decomposition}

The four-corner decomposition for $(1,1)$ is the first member of a larger
family.  If
\[
\alpha=a_1\wedge\cdots\wedge a_p,
\qquad
\beta=b_1\wedge\cdots\wedge b_q,
\]
then expanding every Clifford factor
$C(a)=\mu_a-i_a$ produces $2^{p+q}$ exterior--interior multiplication terms before cancellations.
A term using $r$ contractions among the $p$ left factors and $s$ contractions
among the $q$ right factors has bidegree
\begin{equation*}
(p-2r,q-2s),
\qquad 0\le r\le p,\quad0\le s\le q.
\end{equation*}
After antisymmetrization, terms with the same pair $(r,s)$ combine into a
canonical homogeneous component.  Thus
\begin{equation}\label{eq:Gamma-rs-decomp}
\Gamma(\eta)=\sum_{r=0}^p\sum_{s=0}^q
\Gamma_{r,s}(\eta),
\qquad
\deg\Gamma_{r,s}(\eta)=(p-2r,q-2s).
\end{equation}
Here $\deg\Gamma_{r,s}(\eta)$ denotes the \emph{bidegree shift of the
operator}, not the bidegree of a double form.  Explicitly,
\[
\Gamma_{r,s}(\eta):\D^{u,v}\longrightarrow
\D^{u+p-2r,\,v+q-2s}.
\]
Thus, for $R\in\D^{2,2}$, the central component $\Gamma_{1,1}(R)$ has
$(0,0)$ shift and preserves every $\D^{u,v}$; this does not mean that
$\Gamma_{1,1}(R)$ is an element of $\D^{0,0}$.
For $p=q=1$ these are precisely
$\mu_h,-\Bop_h,-\Bop_h^*,i_h$.

\subsection{Contracted wedges and an explicit formula for every homogeneous channel}
\label{subsec:contracted-wedges}

Contracted wedges give explicit formulas for the homogeneous pieces. We use the classical Graf--Clifford construction \cite{LopesDaRochaGraf}, with the following factorial normalization.

For $\alpha\in\Lambda^pV^*$ and $\xi\in\Lambda^uV^*$, and for
$0\le r\le\min\{p,u\}$, define
\begin{equation}\label{eq:contracted-wedge-def}
{\quad
\alpha\Wedge_r\xi
:=\frac1{r!}\sum_{i_1,\ldots,i_r=1}^n
\bigl(i_{e_{i_r}}\cdots i_{e_{i_1}}\alpha\bigr)\wedge
\bigl(i_{e_{i_r}}\cdots i_{e_{i_1}}\xi\bigr).
\quad}
\end{equation}
For every integer $r>\min\{p,u\}$, set $\alpha\Wedge_r\xi=0$.
The definition is basis-independent by complete metric contraction, and $\Wedge_0=\wedge$. Relative to \cite[Definition~3.1]{LopesDaRochaGraf},
\[
\alpha\Wedge_r\xi=\frac1{r!}\,\alpha\wedge_r^{\mathrm G}\xi.
\]
We retain the convention $C(a)=\mu_a-i_a$ below.

\begin{theorem}
\label{thm:Clifford-contracted-wedge}
Let $\alpha\in\Lambda^pV^*$ and $\xi\in\Lambda^uV^*$.  With the Clifford
convention $C(a)=\mu_a-i_a$ used throughout this paper, the component of
$C(\alpha)\xi$ containing exactly $r$ contractions is
\begin{equation}\label{eq:C-r-Wedge-r}
{\quad
C^{[r]}(\alpha)\xi
=(-1)^{rp-\binom r2}\,\alpha\Wedge_r\xi.
\quad}
\end{equation}
Consequently
\begin{equation}\label{eq:C-full-Wedge}
{\quad
C(\alpha)\xi
=\sum_{r=0}^{\min\{p,u\}}
(-1)^{rp-\binom r2}\,\alpha\Wedge_r\xi.
\quad}
\end{equation}
The same formula holds factorwise for $C_L$ and $C_R$.
\end{theorem}

\begin{proof}
By linearity it suffices to take
$\alpha=e^{j_1}\wedge\cdots\wedge e^{j_p}$ with
$j_1<\cdots<j_p$.  Since the generators are orthonormal and distinct, the
antisymmetrized Clifford action equals the ordered product
\[
C(\alpha)=C(e^{j_1})\cdots C(e^{j_p})
=\prod_{a=1}^p(\mu_{j_a}-i_{j_a}).
\]
Choose $S=\{a_1<\cdots<a_r\}\subset\{1,\ldots,p\}$ for the contraction factors. Moving them to the right contributes
\[
 (-1)^{r+\sum_{t=1}^r(p-a_t-(r-t))}
 =(-1)^{rp-\sum_ta_t-\binom r2+r}.
\]
For the contracted wedge, deleting these indices from $\alpha$ contributes $(-1)^{\sum_t(a_t-t)}$; reversing the remaining contraction order contributes $(-1)^{\binom r2}$. Their product is $(-1)^{\sum_ta_t-r}$. The ratio is therefore $(-1)^{rp-\binom r2}$.
Summing over all $r$-element subsets gives
\[
C^{[r]}(\alpha)\xi
=(-1)^{rp-\binom r2}\frac1{r!}
\sum_{i_1,\ldots,i_r}
(i_{e_{i_r}}\cdots i_{e_{i_1}}\alpha)\wedge
(i_{e_{i_r}}\cdots i_{e_{i_1}}\xi).
\]
Summing over $r$ gives \eqref{eq:C-full-Wedge}; apply the same calculation on each tensor factor.
\end{proof}

\begin{remark}\label{rem:Wedge-low-degree-checks}
For $p=1$, \eqref{eq:C-full-Wedge} is
$C(a)\xi=a\wedge\xi-i_a\xi$.  For $p=2$ one obtains
\[
C(\alpha)\xi
=\alpha\wedge\xi+\alpha\Wedge_1\xi-\alpha\Wedge_2\xi.
\]
With the conventions of Lemma~\ref{lem:C2-decomp},
$\alpha\Wedge_1=-D(\alpha)$ and $\alpha\Wedge_2=i_\alpha$, so this is exactly
$C(\alpha)=\mu_\alpha-D(\alpha)-i_\alpha$.  Thus the contracted-wedge formula
agrees with the independently derived $(2,2)$ channel calculation below.
\end{remark}

\begin{corollary}
\label{cor:Gamma-Wedge-rs}
Let
$\eta=\sum_\lambda\alpha_\lambda\otimes\beta_\lambda\in\D^{p,q}$ and let
$\omega=\xi\otimes\zeta$ be simple.  Then
\begin{equation}\label{eq:Gamma-Wedge-rs}
\begin{aligned}
\Gamma_{r,s}(\eta)(\xi\otimes\zeta)
={}&(-1)^{rp-\binom r2+sq-\binom s2}\\
&\mathrel{\phantom{=}}\times\sum_\lambda
 (\alpha_\lambda\Wedge_r\xi)\otimes
 (\beta_\lambda\Wedge_s\zeta).
\end{aligned}
\end{equation}
The formula extends linearly to arbitrary $\omega\in\D$.
\end{corollary}

\begin{proof}
Apply the contracted-wedge formula of
Theorem~\ref{thm:Clifford-contracted-wedge} separately to the two Clifford
factors.  In
$\Gamma(\eta)=\sum_\lambda C_L(\alpha_\lambda)C_R(\beta_\lambda)$,
select the terms containing $r$ left contractions and $s$ right contractions.
\end{proof}

\begin{corollary}
\label{cor:diagonal-Wedge-channel}
If $\eta=\sum_\lambda\alpha_\lambda\otimes\beta_\lambda\in\D^{p,p}$, then
for every $r$ the signs cancel on the diagonal channel:
\begin{equation}\label{eq:diagonal-Wedge-channel}
{\quad
\Gamma_{r,r}(\eta)(\xi\otimes\zeta)
=\sum_\lambda
(\alpha_\lambda\Wedge_r\xi)\otimes
(\beta_\lambda\Wedge_r\zeta).
\quad}
\end{equation}
In particular, if $\eta\in\D^{2r,2r}$, then $\Gamma_{r,r}(\eta)$ has zero
bidegree shift and is the canonical degree-preserving diagonal channel.
\end{corollary}

\begin{remark}
\label{rem:curvature-power-middle-hierarchy}
For an algebraic curvature tensor $R\in\D^{2,2}$ one has
$R^m\in\D^{2m,2m}$.  Hence
\[
\Gamma_{m,m}(R^m)
\]
is canonically degree-preserving and, by
Corollary~\ref{cor:diagonal-Wedge-channel}, is represented by the paired
contracted wedge $\Wedge_m\otimes\Wedge_m$.  The first member is the central
channel $\Gamma_{1,1}(R)$ identified later in
Theorem~\ref{thm:Gamma-sharp} with $\frac14R\#$.  This observation motivates
the systematic use of contracted wedges in higher Clifford-channel
calculations; no claim is made here that the higher members are classical
sharp products.
\end{remark}

\begin{proposition}
\label{prop:homogeneous-unique}
For fixed $\eta\in\D^{p,q}$, the decomposition
\eqref{eq:Gamma-rs-decomp} is intrinsic.  More precisely,
for every source bidegree $(u,v)$,
\[
 \left.\Gamma_{r,s}(\eta)\right|_{\D^{u,v}}
 =\pi_{u+p-2r,v+q-2s}\circ
   \left.\Gamma(\eta)\right|_{\D^{u,v}},
\]
where $\pi_{a,b}$ is the canonical projection onto $\D^{a,b}$, interpreted as zero for indices outside $\{0,\ldots,n\}$.
\end{proposition}

\begin{proof}
The algebra $\D=\bigoplus_{u,v}\D^{u,v}$ is bigraded.  Hence
$\End(\D)$ decomposes into homogeneous bidegree-shift spaces.  Distinct pairs
$(p-2r,q-2s)$ give distinct shifts, so the corresponding projections are
canonical.  The exterior--interior multiplication expansion shows that no other shifts
occur.  Therefore the sum in \eqref{eq:Gamma-rs-decomp} is not dependent on a
chosen expression of $\eta$ or on a chosen basis.
\end{proof}

\subsection{The \texorpdfstring{$(2,2)$}{(2,2)} case and the degree-preserving component}
\label{subsec:22case}

We now specialize the higher transform to a double form
$R\in\D^{2,2}$.  This includes, in particular, an algebraic curvature tensor.
For a decomposable $2$-form $\alpha=a\wedge b$, define the following degree-zero derivations on the left and right factors by
\begin{align*}
D_L(\alpha)&:=\muL{a}\iL{b}-\muL{b}\iL{a},\\
D_R(\alpha)&:=\muR{a}\iR{b}-\muR{b}\iR{a}.
\end{align*}
These definitions are alternating in $a,b$ and hence extend linearly to
$\Lambda^2V^*$.

Let $\mu^L_\alpha,\mu^R_\alpha$ denote exterior multiplication by $\alpha$
on the left and right factors, and let $i^L_\alpha=(\mu^L_\alpha)^*$,
$i^R_\alpha=(\mu^R_\alpha)^*$.  Direct antisymmetrization of
$C(a)C(b)$ gives the following basic identity.

\begin{lemma}\label{lem:C2-decomp}
For every $\alpha\in\Lambda^2V^*$,
\begin{align}
C_L(\alpha)&=\mu^L_\alpha-D_L(\alpha)-i^L_\alpha,\label{eq:C2L}\\
C_R(\alpha)&=\mu^R_\alpha-D_R(\alpha)-i^R_\alpha.\label{eq:C2R}
\end{align}
\end{lemma}

\begin{proof}
It is enough to take $\alpha=a\wedge b$.  Suppress the subscript $L$ and write
$\mu_a=\muL{a}$ and $i_a=\iL{a}$.  Then
\[
\begin{aligned}
2C_L(\alpha)
={}&(\mu_a-i_a)(\mu_b-i_b)-(\mu_b-i_b)(\mu_a-i_a)\\
={}&(\mu_a\mu_b-\mu_b\mu_a)
 -(\mu_ai_b-\mu_bi_a)\\
&-(i_a\mu_b-i_b\mu_a)
 +(i_ai_b-i_bi_a).
\end{aligned}
\]
Anticommutativity gives
\[
\frac12(\mu_a\mu_b-\mu_b\mu_a)=\mu_a\mu_b=\mu^L_{a\wedge b}.
\]
Similarly,
\[
\frac12(i_ai_b-i_bi_a)=i_ai_b.
\]
With the standard convention
$i^L_{a\wedge b}=\iL{b}\iL{a}$, anticommutativity implies
$i_ai_b=-i^L_{a\wedge b}$.

For the middle terms use
$i_a\mu_b=\langle a,b\rangle\id-\mu_bi_a$ and
$i_b\mu_a=\langle a,b\rangle\id-\mu_ai_b$.  Their scalar parts cancel, and
\[
\begin{aligned}
&-\frac12\{\mu_ai_b-\mu_bi_a+i_a\mu_b-i_b\mu_a\}\\
&\qquad=-(\mu_ai_b-\mu_bi_a)
=-D_L(a\wedge b).
\end{aligned}
\]
This proves \eqref{eq:C2L}.  The right-factor calculation is identical.
\end{proof}

Write
\[
R=\sum_\lambda \alpha_\lambda\otimes\beta_\lambda,
\qquad \alpha_\lambda,\beta_\lambda\in\Lambda^2V^*.
\]
Combining Lemma~\ref{lem:C2-decomp} with
$\Gamma(R)=\sum_\lambda C_L(\alpha_\lambda)C_R(\beta_\lambda)$ yields all nine
homogeneous components explicitly:
\begin{equation}\label{eq:R-nine}
\renewcommand{\arraystretch}{1.25}
\begin{array}{c|ccc}
 & +2 & 0 & -2\\ \hline
+2 & \sum\mu^L_{\alpha_\lambda}\mu^R_{\beta_\lambda}
&-\sum\mu^L_{\alpha_\lambda}D_R(\beta_\lambda)
&-\sum\mu^L_{\alpha_\lambda}i^R_{\beta_\lambda}\\
0&-\sum D_L(\alpha_\lambda)\mu^R_{\beta_\lambda}
&\sum D_L(\alpha_\lambda)D_R(\beta_\lambda)
&\sum D_L(\alpha_\lambda)i^R_{\beta_\lambda}\\
-2&-\sum i^L_{\alpha_\lambda}\mu^R_{\beta_\lambda}
&\sum i^L_{\alpha_\lambda}D_R(\beta_\lambda)
&\sum i^L_{\alpha_\lambda}i^R_{\beta_\lambda}
\end{array}
\end{equation}
where the row records the left degree shift and the column the right degree
shift.  These pairs are operator shifts, not bidegrees of double forms.  In the notation of \eqref{eq:Gamma-rs-decomp}, the unique component
preserving both bidegrees is therefore
\begin{equation}\label{eq:Gamma11R}
{\quad
\Gamma_{1,1}(R)
=\sum_\lambda D_L(\alpha_\lambda)D_R(\beta_\lambda).
\quad}
\end{equation}
The formula is independent of the chosen decomposition of $R$, since it is the
$(0,0)$-homogeneous projection of the intrinsically defined operator
$\Gamma(R)$.

\subsection{Identification with the sharp product}
\label{subsec:sharp}

Hamilton's curvature sharp operation \cite{HamiltonFour,BoehmWilking,Schmidt} has a Clifford-commutator extension to double forms \cite{LabbiLaplacians}. We use the latter normalization throughout this section. For the geometric applications we fix
\[
 R(x,y;z,w)=g\bigl(R^{\mathrm{Riem}}(x,y)w,z\bigr),\qquad
 \operatorname{Ric}=cR,\qquad\operatorname{scal}=c^2R,
\]
where $R^{\mathrm{Riem}}(X,Y)=[\nabla_X,\nabla_Y]-\nabla_{[X,Y]}$. Thus a space form of sectional curvature $\kappa$ has $R=\kappa g^2/2$.
Let the exterior algebra $\Lambda V^*$ be identified with the Clifford algebra
through the convention
\[
a\mathbin{\cdot_{\mathrm{Cl}}}\eta=a\wedge\eta-i_{a^\sharp}\eta,
\qquad a\in V^*.
\]
Write
\[
[\alpha,\beta]_{\mathrm{Cl}}
:=\alpha\mathbin{\cdot_{\mathrm{Cl}}}\beta-\beta\mathbin{\cdot_{\mathrm{Cl}}}\alpha
\]
for the Clifford commutator. The sharp product on double forms, introduced by Labbi \cite{LabbiLaplacians}, is defined on simple tensors by
\begin{equation}\label{eq:sharp-def}
(\alpha_1\otimes\alpha_2)\#(\beta_1\otimes\beta_2)
:=[\alpha_1,\beta_1]_{\mathrm{Cl}}
 \otimes[\alpha_2,\beta_2]_{\mathrm{Cl}},
\end{equation}
and extended bilinearly. The product is symmetric. On $\D^{2,2}$ it is related to the bracket-lift product $\#_{\mathrm H}$ of \cite{BoehmWilking} by
\[
 R\#S=8(R\#_{\mathrm H}S).
\]
Indeed, equip $\mathfrak{so}(V)$ with $\langle A,B\rangle=-\tfrac12\operatorname{tr}(AB)$. In an orthonormal basis $A_a$, the bracket-lift expansion is
\[
 R\#_{\mathrm H}S=\frac12\sum_{a,b,c,d}R_{ab}S_{cd}
 [A_a,A_c]\otimes[A_b,A_d],
\]
whereas each Clifford commutator contributes a factor $2$ in \eqref{eq:sharp-def}.

For a $2$-form $\alpha$, Labbi's Clifford-adjoint formula reads
\begin{equation}\label{eq:Labbi-ad-rho}
[\alpha,\eta]_{\mathrm{Cl}}=2\rho(\alpha)\eta,
\qquad \eta\in\Lambda V^*,
\end{equation}
where $\rho$ is defined directly as follows.  The metric $g$ determines the canonical identification
\[
\rho:\Lambda^2V\longrightarrow\mathfrak{so}(V),
\]
which on decomposable bivectors is given by
\[
\rho(x\wedge y)(z)=g(x,z)y-g(y,z)x,
\qquad x,y,z\in V,
\]
and is extended linearly to $\Lambda^2V$.  We use the same notation for the induced degree-zero derivation of the exterior algebra.\footnote{Explicitly, for $A\in\Lambda^2V$ and a decomposable element $v_1\wedge\cdots\wedge v_p\in\Lambda^pV$,
\[
\rho(A)(v_1\wedge\cdots\wedge v_p)
=\sum_{j=1}^p v_1\wedge\cdots\wedge \rho(A)v_j\wedge\cdots\wedge v_p.
\]}
Equivalently, for $A\in\Lambda^2V$, $\xi\in\Lambda^pV$, and $\zeta\in\Lambda^qV$,
\[
\rho(A)(\xi\wedge\zeta)
=\rho(A)\xi\wedge\zeta+\xi\wedge\rho(A)\zeta,
\]
so $\rho(A)$ has degree zero.

For $\alpha\in\Lambda^2V^*$ let $\alpha^\sharp\in\mathfrak{so}(V)$ be determined by
\[
g(\alpha^\sharp x,y)=\alpha(x,y).
\]
Under the metric identification, $\alpha^\sharp$ is precisely the skew endomorphism described above.  Its induced action on $1$-forms is
\[
(\rho(\alpha)u)(x)=-u(\alpha^\sharp x),
\qquad u\in V^*,
\]
and we again denote by $\rho(\alpha)$ its unique extension to $\Lambda V^*$ as a degree-zero derivation.  Thus, for $\eta\in\Lambda^pV^*$,
\begin{equation}\label{eq:rho-p-form}
{
(\rho(\alpha)\eta)(x_1,\ldots,x_p)
=-\sum_{r=1}^p
\eta(x_1,\ldots,\alpha^\sharp x_r,\ldots,x_p).
}
\end{equation}
In particular, on $2$-forms,
\[
(\rho(\alpha)\eta)(x,y)
=-\eta(\alpha^\sharp x,y)-\eta(x,\alpha^\sharp y).
\]
This convention separates the metric identification
$\Lambda^2V^*\simeq\mathfrak{so}(V)$ from the induced action of
$\mathfrak{so}(V)$ on covariant forms.

With our notation, writing $\mu_a$ for exterior multiplication by $a$ on
$\Lambda V^*$ and $i_a:=i_{a^\sharp}$ for interior multiplication, we set
\[
D(a\wedge b)=\mu_a i_b-\mu_b i_a.
\]
This is the one-factor operator underlying the previously defined operators
$D_L$ and $D_R$.  Both $D(\alpha)$ and $\rho(\alpha)$ are degree-zero
derivations of $\Lambda V^*$, so it suffices to compare them on $1$-forms.  A
direct calculation gives
\begin{equation}\label{eq:D-minus-rho}
D(\alpha)=-\rho(\alpha).
\end{equation}
The simultaneous minus signs on the two tensor factors will cancel below.

\begin{theorem}
\label{thm:Gamma-sharp}
For every $R\in\D^{2,2}$ and every double form $\omega\in\D$,
\begin{equation*}
{\quad
\Gamma_{1,1}(R)(\omega)=\frac14\,R\#\omega.
\quad}
\end{equation*}
No symmetry or Bianchi assumption on $R$ is required.  Thus the sharp
product is precisely four times the bidegree-preserving component of the
higher left--right Clifford transform.
\end{theorem}

\begin{proof}
By bilinearity it suffices to take
\[
R=\alpha\otimes\beta\in\D^{2,2},
\qquad
\omega=u\otimes v.
\]
Definition~\eqref{eq:sharp-def}, followed by
\eqref{eq:Labbi-ad-rho} and \eqref{eq:D-minus-rho}, gives
\[
\begin{aligned}
R\#\omega
&=[\alpha,u]_{\mathrm{Cl}}\otimes
  [\beta,v]_{\mathrm{Cl}}\\
&=4\rho(\alpha)u\otimes\rho(\beta)v\\
&=4D(\alpha)u\otimes D(\beta)v\\
&=4D_L(\alpha)D_R(\beta)(u\otimes v).
\end{aligned}
\]
On the other hand, the $(0,0)$-homogeneous component computed in
\eqref{eq:Gamma11R} is
\[
\Gamma_{1,1}(\alpha\otimes\beta)
=D_L(\alpha)D_R(\beta).
\]
Therefore $R\#\omega=4\Gamma_{1,1}(R)\omega$ for simple tensors, and
bilinearity proves the result for arbitrary $R\in\D^{2,2}$ and arbitrary
$\omega\in\D$.
\end{proof}

\begin{remark}
\label{rem:Labbi-inspiration-channels}
Labbi's formula \cite{LabbiLaplacians} motivated the central-channel identification. The transform and its homogeneous decomposition are defined independently; Theorem~\ref{thm:Gamma-sharp} identifies a known sharp operation within that decomposition.
\end{remark}

\begin{proposition}
\label{prop:metric-square-central-channel}
Let $\omega\in\D^{p,q}$.  With the conventions of
\eqref{eq:S}--\eqref{eq:Sadj}, the central channel of the metric square
satisfies
\begin{equation}\label{eq:metric-square-channel-p}
{
\Gamma_{1,1}(g^2)\omega
=2g\cdot c\omega-2p\,\omega+2\Bop\Bop^*\omega.
}
\end{equation}
Equivalently,
\begin{equation}\label{eq:metric-square-channel-q}
{
\Gamma_{1,1}(g^2)\omega
=2g\cdot c\omega-2q\,\omega+2\Bop^*\Bop\omega.
}
\end{equation}
Consequently, by Theorem~\ref{thm:Gamma-sharp},
\begin{equation}\label{eq:Labbi-cor44-our-notation}
{
\frac12 g^2\#\omega
=4g\cdot c\omega-4q\,\omega+4\Bop^*\Bop\omega,
}
\end{equation}
which is Labbi's metric-square sharp identity
\cite[Corollary~4.4]{LabbiLaplacians}, written in our notation.
\end{proposition}

\begin{proof}
Let $e_1,\ldots,e_n$ be an orthonormal basis, with dual basis
$e^1,\ldots,e^n$. Write $e^{ij}:=e^i\wedge e^j$ and
$D_{ij}^L:=D_L(e^{ij})$, $D_{ij}^R:=D_R(e^{ij})$.
These definitions apply to all indices: $e^{ji}=-e^{ij}$, $e^{ii}=0$, and likewise $D_{ji}^{L,R}=-D_{ij}^{L,R}$, $D_{ii}^{L,R}=0$.
The forms $e^{ij}$ with $i<j$ are an orthonormal basis of $\Lambda^2V^*$. Since
\[
g^2=2\sum_{i<j}e^{ij}\otimes e^{ij},
\]
formula \eqref{eq:Gamma11R} gives, without using the sharp product,
\begin{equation}\label{eq:g2-Gamma11-D}
\Gamma_{1,1}(g^2)
=2\sum_{i<j}D_{ij}^LD_{ij}^R.
\end{equation}
For brevity write
$\mu_i^L=\muL{e^i}$, $i_i^L=\iL{e^i}$ and similarly on the right.  By
definition,
\[
D_{ij}^L=\mu_i^Li_j^L-\mu_j^Li_i^L,
\qquad
D_{ij}^R=\mu_i^Ri_j^R-\mu_j^Ri_i^R.
\]
The summand is invariant under interchanging $i$ and $j$, so
\eqref{eq:g2-Gamma11-D} may be written as the unrestricted sum
\[
\Gamma_{1,1}(g^2)
=\sum_{i,j}
(\mu_i^Li_j^L-\mu_j^Li_i^L)
(\mu_i^Ri_j^R-\mu_j^Ri_i^R).
\]
Expanding and relabelling dummy indices yields
\begin{equation}\label{eq:g2-central-cross}
\Gamma_{1,1}(g^2)
=2\mu_gc-2X,
\qquad
X:=\sum_{i,j}\mu_i^Li_j^L\mu_j^Ri_i^R.
\end{equation}
Indeed, the two diagonal terms combine to
$2\mu_gc$, while the two crossed terms combine to $-2X$.

We now identify $X$ using only the canonical anticommutation relations and
our Bianchi operators.  Since left and right factor operators commute,
\[
\begin{aligned}
\Bop\Bop^*
&=\sum_{i,j}\mu_i^Li_i^R i_j^L\mu_j^R\\
&=\sum_{i,j}\mu_i^Li_j^L i_i^R\mu_j^R\\
&=\sum_i\mu_i^Li_i^L
 -\sum_{i,j}\mu_i^Li_j^L\mu_j^Ri_i^R.
\end{aligned}
\]
On $\D^{p,q}$ the first sum is $p\,\Id$.  Hence
\[
X=p\,\Id-\Bop\Bop^*.
\]
Substitution into \eqref{eq:g2-central-cross} proves
\eqref{eq:metric-square-channel-p}.

The second form follows either by the same calculation with the factors
reversed or, more economically, from Proposition~\ref{prop:S-sl2}:
\[
\Bop\Bop^*-\Bop^*\Bop=(p-q)\Id
\quad\hbox{on }\D^{p,q}.
\]
Therefore
\[
-2p\Id+2\Bop\Bop^*
=-2q\Id+2\Bop^*\Bop,
\]
which proves \eqref{eq:metric-square-channel-q}.  Finally,
Theorem~\ref{thm:Gamma-sharp} gives
$\Gamma_{1,1}(g^2)=\frac14g^2\#$; multiplying
\eqref{eq:metric-square-channel-q} by $2$ gives
\eqref{eq:Labbi-cor44-our-notation}.
\end{proof}

\begin{remark}
\label{rem:retain-Labbi-computation}
Equation~\eqref{eq:Labbi-cor44-our-notation} is classical \cite{LabbiLaplacians}. The direct calculation above fixes the channel normalization; the difference of \eqref{eq:metric-square-channel-p} and \eqref{eq:metric-square-channel-q} is exactly $[\Bop,\Bop^*]=(p-q)\Id$.
\end{remark}

\begin{proposition}
\label{prop:lowest-curvature-contractions}
Let $R\in\D^{2,2}$ be an algebraic curvature tensor.  With the indexing of
\eqref{eq:Gamma-rs-decomp}, its maximally contracting homogeneous component is
\[
\Gamma_{2,2}(R)=i_R,
\qquad
\deg\Gamma_{2,2}(R)=(-2,-2).
\]
Then
\begin{equation}\label{eq:lowest-scalar}
{\quad
\Gamma_{2,2}(R)(g^2)=\operatorname{scal}(R),
\quad}
\end{equation}
and
\begin{equation}\label{eq:lowest-Einstein}
{\quad
\Gamma_{2,2}(R)(g^3)
=3\,\operatorname{scal}(R)g-6\,\operatorname{Ric}(R)
=-6G_R,
\quad}
\end{equation}
where
\[
G_R:=\operatorname{Ric}(R)-\frac12\operatorname{scal}(R)g
\]
is the Einstein tensor associated with $R$.  Consequently,
\begin{equation}\label{eq:lowest-Ricci-reconstruction}
\operatorname{Ric}(R)
=\frac12\Gamma_{2,2}(R)(g^2)g
-\frac16\Gamma_{2,2}(R)(g^3).
\end{equation}
\end{proposition}

\begin{proof}
The bottom-right entry of \eqref{eq:R-nine} gives
\[
\Gamma_{2,2}(R)
=\sum_\lambda i^L_{\alpha_\lambda}i^R_{\beta_\lambda}
=i_R,
\]
the adjoint of exterior multiplication by $R$.  Hence
\[
\Gamma_{2,2}(R)(g^2)
=\langle g^2,R\rangle
=c^2R
=\operatorname{scal}(R),
\]
with the standard unnormalized double-form convention used throughout the
paper.

For the second identity, let $h\in\D^{1,1}$ be symmetric.  By adjunction,
\[
\langle i_R(g^3),h\rangle=\langle g^3,R\cdot h\rangle=c^3(R\cdot h).
\]
Diagonalize $h=\sum_i h_i e^i\otimes e^i$ and put
$K_{ij}=R(e_i,e_j;e_i,e_j)$. Then
\[
\begin{aligned}
 c^3(R\cdot h)
 &=6\sum_{i<j}K_{ij}(ch-h_i-h_j)\\
 &=3(c^2R)(ch)-6\langle cR,h\rangle,
\end{aligned}
\]
since $c^2R=2\sum_{i<j}K_{ij}$ and $(cR)_{ii}=\sum_{j\ne i}K_{ij}$.
Since $cR=\operatorname{Ric}(R)$, $c^2R=\operatorname{scal}(R)$, and
$ch=\langle g,h\rangle$, we obtain
\[
\langle i_R(g^3),h\rangle
=\left\langle
3\operatorname{scal}(R)g-6\operatorname{Ric}(R),h
\right\rangle.
\]
Both sides are symmetric $(1,1)$-double forms, so \eqref{eq:lowest-Einstein}
follows.  Equation \eqref{eq:lowest-Ricci-reconstruction} is an immediate
rearrangement.
\end{proof}

\begin{remark}
\label{rem:lowest-classical}
The contraction formulas are classical \cite{LabbiDoubleForms,LabbiAlgebraic}. The three diagonal positions in the transform are
\[
 \Gamma_{0,0}(R)=\mu_R,\qquad
 \Gamma_{1,1}(R)=\frac14R\#,\qquad
 \Gamma_{2,2}(R)=i_R,
\]
where $\mu_R(\omega)=R\cdot\omega$ is multiplication by the double form $R$.
\end{remark}

\subsection{Recovery of the orthogonal curvature decomposition from the extreme channels}
\label{subsec:orthogonal-curvature-channel-recovery}

The highest channel recovers $R$, and the lowest recovers its traces. Together they give the classical orthogonal curvature decomposition.

Throughout this subsection $n\geq4$.  Set
\begin{equation}\label{eq:SR-HR-def}
 S_R:=\Gamma_{2,2}(R)(g^2),
 \qquad
 H_R:=\Gamma_{2,2}(R)(g^3).
\end{equation}
By Proposition~\ref{prop:lowest-curvature-contractions},
\begin{equation}\label{eq:SR-HR-classical}
 S_R=\operatorname{scal}(R),
 \qquad
 H_R=3S_Rg-6\operatorname{Ric}(R),
\end{equation}
and therefore
\begin{equation}\label{eq:Ric-from-SH}
 \operatorname{Ric}(R)=\frac12S_Rg-\frac16H_R.
\end{equation}
On the other hand,
\begin{equation}\label{eq:R-from-highest-channel}
 \Gamma_{0,0}(R)(1)=\mu_R(1)=R.
\end{equation}
Thus the highest channel supplies $R$ itself, while the lowest channel
supplies its traces.

\begin{theorem}
\label{thm:channel-orthogonal-curvature-decomposition}
Let $R\in\D^{2,2}$ be an algebraic curvature tensor on an $n$-dimensional
Euclidean vector space, $n\geq4$.  In the unnormalized double-form convention
of this paper, write
\[
 R=W_R+R_{\operatorname{Ric}_0}+R_{\mathrm{sc}}
\]
for the orthogonal decomposition into the Weyl, traceless-Ricci and scalar
curvature summands.  Then each summand is recovered from the extreme Clifford
channels.  More precisely,
\begin{align}
 R_{\mathrm{sc}}
 &=\frac{S_R}{2n(n-1)}g^2,
 \label{eq:channel-scalar-component}\\
 R_{\operatorname{Ric}_0}
 &=\frac{S_R}{2n}g^2
   -\frac{1}{6(n-2)}g\cdot H_R,
 \label{eq:channel-traceless-Ricci-component}\\
 W_R
 &=\Gamma_{0,0}(R)(1)
   +\frac{1}{6(n-2)}g\cdot H_R
   -\frac{S_R}{2(n-1)}g^2.
 \label{eq:channel-Weyl-component}
\end{align}
Equivalently, after restoring the channel notation,
\begin{equation}\label{eq:channel-Weyl-full}
{
\begin{aligned}
W_R={}&\Gamma_{0,0}(R)(1)
+\frac{1}{6(n-2)}g\cdot\Gamma_{2,2}(R)(g^3)\\
&-\frac{1}{2(n-1)}\Gamma_{2,2}(R)(g^2)\,g^2.
\end{aligned}}
\end{equation}
Hence the pair of extreme components $(\Gamma_{0,0},\Gamma_{2,2})$ recovers
all three irreducible summands of the orthogonal curvature decomposition.
\end{theorem}

\begin{proof}
From
\[
 c(g\cdot\omega)=g\cdot c\omega+(n-p-q)\omega
\]
with $\omega=g\in\D^{1,1}$, and from $c(g)=n$, one obtains
\begin{equation}\label{eq:cg2-audit}
 c(g^2)=2(n-1)g,
 \qquad
 c^2(g^2)=2n(n-1).
\end{equation}
Therefore the scalar-curvature summand having scalar curvature $S_R$ is
\[
 R_{\mathrm{sc}}=\frac{S_R}{2n(n-1)}g^2,
\]
which proves \eqref{eq:channel-scalar-component}.

Next define
\[
 \operatorname{Ric}_0
 :=\operatorname{Ric}(R)-\frac{S_R}{n}g.
\]
Using \eqref{eq:Ric-from-SH},
\begin{equation}\label{eq:Ric0-from-SH}
 \operatorname{Ric}_0
 =\frac{n-2}{2n}S_Rg-\frac16H_R.
\end{equation}
Since $c\operatorname{Ric}_0=0$, the fundamental metric--contraction identity
gives
\[
 c(g\cdot\operatorname{Ric}_0)=(n-2)\operatorname{Ric}_0.
\]
Thus the curvature tensor with traceless Ricci contraction
$\operatorname{Ric}_0$ is
\[
 R_{\operatorname{Ric}_0}
 =\frac1{n-2}g\cdot\operatorname{Ric}_0.
\]
Substitution of \eqref{eq:Ric0-from-SH} yields
\[
 R_{\operatorname{Ric}_0}
 =\frac{S_R}{2n}g^2
  -\frac{1}{6(n-2)}g\cdot H_R,
\]
which is \eqref{eq:channel-traceless-Ricci-component}.

Finally, the highest channel gives $R=\Gamma_{0,0}(R)(1)$.  Hence
\[
 W_R
 =R-R_{\operatorname{Ric}_0}-R_{\mathrm{sc}}.
\]
Using the two formulas already established,
\begin{align*}
W_R
&=R+\frac{1}{6(n-2)}g\cdot H_R
 -\left(\frac1{2n}+\frac1{2n(n-1)}\right)S_Rg^2\\
&=R+\frac{1}{6(n-2)}g\cdot H_R
 -\frac{S_R}{2(n-1)}g^2,
\end{align*}
which proves \eqref{eq:channel-Weyl-component} and
\eqref{eq:channel-Weyl-full}.

To verify $cW_R=0$, use
\eqref{eq:SR-HR-classical},
\[
 cH_R=3nS_R-6S_R=3(n-2)S_R.
\]
Since $H_R\in\D^{1,1}$,
\begin{equation}\label{eq:cgHR-audit}
 c(g\cdot H_R)
 =g\cdot cH_R+(n-2)H_R
 =3(n-2)S_Rg+(n-2)H_R.
\end{equation}
Together with
\[
 cR=\frac12S_Rg-\frac16H_R
\]
and \eqref{eq:cg2-audit}, this gives
\begin{align*}
cW_R
={}&\left(\frac12S_Rg-\frac16H_R\right)
+\frac1{6(n-2)}
 \left(3(n-2)S_Rg+(n-2)H_R\right)\\
&-\frac{S_R}{2(n-1)}\,2(n-1)g\\
={}&0.
\end{align*}
Thus the remainder is the Weyl component.
\end{proof}

\begin{remark}
\label{rem:channel-orthogonal-classical}
The orthogonal decomposition is classical. Schmidt \cite[\S 3.2.1.2, Lemma~3.2.1.5 and Corollary~3.2.1.7, pp.~54--58]{Schmidt} realizes its scalar, traceless-Ricci and Weyl summands as eigenspaces of a self-adjoint Ricci-induced operator. Theorem~\ref{thm:channel-orthogonal-curvature-decomposition} expresses these same summands through the extreme channels, using $\Gamma_{0,0}(R)(1)=R$ and the traces obtained from $\Gamma_{2,2}(R)$.
\end{remark}

\subsection{A working example: the Kodaira--Thurston metric}
\label{subsec:kodaira-thurston-example}

We compute the central channel on the classical Kodaira--Thurston metric.

Let $M$ be the Kodaira--Thurston manifold.  Choose a standard left-invariant
coframe $e^1,e^2,e^3,e^4$ satisfying
\begin{equation}\label{eq:KT-structure}
de^1=de^2=de^3=0,\qquad de^4=e^{12},
\end{equation}
and the metric
\[
g=(e^1)^2+(e^2)^2+(e^3)^2+(e^4)^2.
\]
This standard presentation may be found, for example, in
\cite{BuzanoFinoVezzoni}.  Write $(e_1,e_2,e_3,e_4)$ for the dual orthonormal
frame.  Equation~\eqref{eq:KT-structure} gives
\[
[e_1,e_2]=-e_4,
\]
with all other brackets zero.  The Koszul formula
\[
2\langle\nabla_XY,Z\rangle
=\langle[X,Y],Z\rangle-\langle[Y,Z],X\rangle
 +\langle[Z,X],Y\rangle
\]
then gives the only nonzero covariant derivatives, up to the displayed
symmetries,
\[
\nabla_{e_1}e_2=-\frac12e_4,\qquad
\nabla_{e_2}e_1=\frac12e_4,
\]
\[
\nabla_{e_1}e_4=\nabla_{e_4}e_1=\frac12e_2,
\qquad
\nabla_{e_2}e_4=\nabla_{e_4}e_2=-\frac12e_1.
\]
Substituting the connection coefficients into
\[
R^{\mathrm{Riem}}(X,Y)Z=\nabla_X\nabla_YZ-\nabla_Y\nabla_XZ-\nabla_{[X,Y]}Z
\]
gives, with the curvature convention of this paper,
\[
K_{12}=-\frac34,\qquad K_{14}=K_{24}=\frac14,
\]
while all sectional curvatures involving $e_3$ vanish.  More precisely, direct
substitution shows that, relative to the ordered basis
\[
e^{12},e^{13},e^{14},e^{23},e^{24},e^{34}
\]
of $\Lambda^2V^*$, the curvature-operator matrix is
\[
\operatorname{diag}\!\left(-\frac34,0,\frac14,0,\frac14,0\right).
\]
The off-diagonal entries vanish. Thus
\begin{equation}\label{eq:KT-curvature}
R=-\frac34\,e^{12}\otimes e^{12}
  +\frac14\,e^{14}\otimes e^{14}
  +\frac14\,e^{24}\otimes e^{24}.
\end{equation}

We now compute the sharp reaction term using only Clifford commutators.  Among
the three $2$-forms occurring in~\eqref{eq:KT-curvature}, the nonzero
commutators are, up to sign,
\[
[e^{14},e^{24}]_{\rm Cl}=\pm2e^{12},\qquad
[e^{12},e^{24}]_{\rm Cl}=\pm2e^{14},\qquad
[e^{12},e^{14}]_{\rm Cl}=\pm2e^{24}.
\]
The signs disappear after taking the tensor product of the two identical
commutators.  Since both ordered pairs $(\alpha,\beta)$ and
$(\beta,\alpha)$ occur in $R\#R$, Definition~\eqref{eq:sharp-def} gives
\begin{equation}\label{eq:KT-sharp}
R\#R
=\frac12\,e^{12}\otimes e^{12}
 -\frac32\,e^{14}\otimes e^{14}
 -\frac32\,e^{24}\otimes e^{24}.
\end{equation}
For example, the $e^{12}\otimes e^{12}$ coefficient is
\[
2\cdot4\left(\frac14\right)\left(\frac14\right)=\frac12,
\]
and the $e^{14}\otimes e^{14}$ coefficient is
\[
2\cdot4\left(-\frac34\right)\left(\frac14\right)=-\frac32.
\]
The third coefficient is identical.

Theorem~\ref{thm:Gamma-sharp} now reads, concretely,
\begin{equation}\label{eq:KT-Gamma}
{
\Gamma_{1,1}(R)R
=\frac18\,e^{12}\otimes e^{12}
 -\frac38\,e^{14}\otimes e^{14}
 -\frac38\,e^{24}\otimes e^{24}.}
\end{equation}
Equivalently,
\[
\Gamma_{1,1}(R)R
=\frac14R\#R
\]
is the sharp reaction contribution in the convention fixed above.

\begin{remark}
The example also explains why the factor $1/4$ in
Theorem~\ref{thm:Gamma-sharp} is geometrically visible rather than merely
formal: each Clifford commutator of $2$-forms contributes a factor $2$, so the
double commutator contributing to $R\#R$ carries the factor $4$ that is removed
by $\Gamma_{1,1}$.
\end{remark}

\subsection{A geometric detection example: \texorpdfstring{$S^2\times S^2$}{S2 x S2}}
\label{subsec:s2xs2-example}

The next example separates Ricci contraction from full curvature.

Let
\[
M=S^2(\kappa)\times S^2(\kappa),
\qquad g=g_1+g_2,
\]
where each factor has constant sectional curvature $\kappa>0$.  With the
normalization used throughout the paper,
\begin{equation}\label{eq:S2xS2-curvature}
R=\frac{\kappa}{2}\bigl(g_1^2+g_2^2\bigr).
\end{equation}
The classical product-curvature formulas give
\begin{equation}\label{eq:S2xS2-Ricci}
\operatorname{Ric}(R)=\kappa(g_1+g_2)=\kappa g,
\qquad
\operatorname{scal}(R)=4\kappa.
\end{equation}
Thus the product metric is Einstein.  Proposition~\ref{prop:lowest-curvature-contractions}
shows that the lowest Clifford channel detects exactly this contracted
information:
\[
\Gamma_{2,2}(R)(g^2)=4\kappa,
\qquad
\Gamma_{2,2}(R)(g^3)=6\kappa g,
\qquad
G_R=-\kappa g.
\]

The full curvature, however, contains information lost by Ricci contraction.
If $\sigma$ is a $2$-plane tangent to either sphere, then
\[
K(\sigma)=\kappa,
\]
whereas every mixed plane spanned by one vector from each factor has
\[
K(\sigma)=0.
\]
Hence the metric is Einstein but does not have constant sectional curvature.
The Weyl component records information lost under Ricci contraction.

Here the extreme and central channels are
\[
 \Gamma_{2,2}(R)=i_R,\qquad
 \Gamma_{0,0}(R)(1)=\frac\kappa2(g_1^2+g_2^2),\qquad
 \Gamma_{1,1}(R)=\frac14R\#.
\]
The lowest channel recovers the Einstein contraction, while the highest retains the vanishing mixed-plane curvature.

\begin{remark}
Both geometries are classical. The first example computes the sharp term; the second distinguishes Einstein contraction data from sectional curvature.
\end{remark}

\begin{remark}
The factor $1/4$ follows from $[\alpha,\eta]_{\mathrm{Cl}}=2\rho(\alpha)\eta$ on each factor. Theorem~\ref{thm:Gamma-sharp} locates the sharp product on double forms in the homogeneous transform; it does not redefine that product.
\end{remark}

\subsection{The same-factor Clifford channel and the Weitzenb\"ock curvature term}
\label{subsec:same-factor-weitzenbock}

Throughout this subsection $R$ is an algebraic curvature tensor.
The mixed degree-preserving channel of Theorem~\ref{thm:Gamma-sharp} places
one $D$-operator on each exterior factor and gives
\[
\Gamma_{1,1}(R)=\frac14R\#.
\]
For comparison, place both degree-preserving operators on one factor.

Choose an orthonormal coframe $e^1,\ldots,e^n$. Put
\[
 e^{ij}=e^i\wedge e^j,\qquad D_{ij}:=D(e^{ij}),\qquad
 R_{ij,kl}:=\langle\mathcal R e^{ij},e^{kl}\rangle.
\]
Thus $R_{ij,kl}=R(e_i,e_j;e_k,e_l)$ for all indices, with
$R_{ji,kl}=-R_{ij,kl}$ and $R_{ij,lk}=-R_{ij,kl}$. In particular,
\[
 R=\sum_{i<j}\sum_{k<l}R_{ij,kl}\,e^{ij}\otimes e^{kl}.
\]
On $\Lambda^pV^*$ define the \emph{same-factor $D$-channel}
\begin{equation}\label{eq:same-factor-D-channel}
{
\mathcal W_p^D(R):=-\sum_{i<j}\sum_{k<l}R_{ij,kl}D_{ij}D_{kl}.
}
\end{equation}

Recall that $D_{ij}=-\rho(e^{ij})$, where $\rho(e^{ij})$ denotes the induced action of $(e^{ij})^\sharp\in\mathfrak{so}(V)$ on exterior forms, with
$g((e^{ij})^\sharp X,Y)=e^{ij}(X,Y)$. Thus
\begin{equation}\label{eq:DA-slot-formula}
(D_{ij}\omega)(X_1,\ldots,X_p)
=\sum_{r=1}^p\omega(X_1,\ldots,(e^{ij})^\sharp X_r,\ldots,X_p).
\end{equation}

\subsubsection{Degree one: direct computation}

For the dual orthonormal frame $e_1,\ldots,e_n$,

\begin{equation}\label{eq:Eij-sharp-WD}
(e^{ij})^\sharp e_k=\delta_{ik}e_j-\delta_{jk}e_i,
\end{equation}
and consequently, for the dual basis,
\begin{equation}\label{eq:DEij-oneform-WD}
D_{ij}e^m=\delta_{jm}e^i-\delta_{im}e^j.
\end{equation}
For a general one-form $\alpha$,
\[
(D_{ij}\alpha)(X)=\alpha((e^{ij})^\sharp X),
\]
so that
\[
(D_{ij}D_{kl}\alpha)(X)
=\alpha((e^{kl})^\sharp (e^{ij})^\sharp X).
\]
Therefore
\begin{equation}\label{eq:WD1-precontraction}
(\mathcal W_1^D(R)\alpha)(X)
=-\alpha\!\left(
\sum_{i<j}\sum_{k<l}R_{ij,kl}
(e^{kl})^\sharp (e^{ij})^\sharp X\right).
\end{equation}
The contraction in parentheses may be evaluated without invoking any
Weitzenb\"ock identity.  From \eqref{eq:Eij-sharp-WD}, expanding the two
skew endomorphisms and using the curvature symmetries in the two index pairs
gives
\begin{equation}\label{eq:WD-Ricci-contraction}
\sum_{i<j}\sum_{k<l}R_{ij,kl}
(e^{kl})^\sharp (e^{ij})^\sharp X=-\operatorname{Ric}^\sharp X.
\end{equation}
Substitution into \eqref{eq:WD1-precontraction} yields the first result of the
same-factor calculation:
\begin{equation}\label{eq:WD1-result}
{
(\mathcal W_1^D(R)\alpha)(X)
=\alpha(\operatorname{Ric}^\sharp X),
\qquad
\mathcal W_1^D(R)=\operatorname{Ric}.
}
\end{equation}

\paragraph{Comparison in degree one.}
The established Weitzenb\"ock curvature term on one-forms is
$\mathcal N_1(R)=\operatorname{Ric}$.  Comparing this established formula
with the independently obtained \eqref{eq:WD1-result} gives
\[
\mathcal W_1^D(R)=\mathcal N_1(R).
\]

\subsubsection{Degree two: direct computation}

The degree-zero operator $D_{ij}$ is a derivation of the exterior algebra.  For
one-forms $\alpha,\beta$,
\[
D_{ij}(\alpha\wedge\beta)
=D_{ij}\alpha\wedge\beta+\alpha\wedge D_{ij}\beta.
\]
Applying $D_{ij}$ after $D_{kl}$ gives the four terms
\begin{align}\label{eq:DADB-four-terms}
D_{ij}D_{kl}(\alpha\wedge\beta)
={}&D_{ij}D_{kl}\alpha\wedge\beta
+D_{kl}\alpha\wedge D_{ij}\beta\\
&+D_{ij}\alpha\wedge D_{kl}\beta
+\alpha\wedge D_{ij}D_{kl}\beta.\nonumber
\end{align}
After multiplication by $-R_{ij,kl}$ and summation, the first and fourth terms
are already determined by the degree-one calculation:
\begin{align*}
-\sum_{i<j}\sum_{k<l}R_{ij,kl}D_{ij}D_{kl}\alpha\wedge\beta
&=(\operatorname{Ric}\alpha)\wedge\beta,\\
-\sum_{i<j}\sum_{k<l}R_{ij,kl}\alpha\wedge D_{ij}D_{kl}\beta
&=\alpha\wedge(\operatorname{Ric}\beta).
\end{align*}
Together they form the induced Ricci action
\begin{equation}\label{eq:Ric-induced-two}
\operatorname{Ric}^{[2]}(\alpha\wedge\beta)
=(\operatorname{Ric}\alpha)\wedge\beta
+\alpha\wedge(\operatorname{Ric}\beta).
\end{equation}

Since $R_{ij,kl}=R_{kl,ij}$, the two middle terms give
\begin{equation}\label{eq:WD2-middle}
-2\sum_{i<j}\sum_{k<l}R_{ij,kl}D_{ij}\alpha\wedge D_{kl}\beta.
\end{equation}
Using \eqref{eq:DEij-oneform-WD},
\begin{align}\label{eq:WD2-four-expansion}
&\sum_{i<j}\sum_{k<l}R_{ij,kl}
D_{ij}\alpha\wedge D_{kl}\beta\\
={}&\sum_{i<j}\sum_{k<l}R_{ij,kl}
\bigl(\alpha(e_j)e^i-\alpha(e_i)e^j\bigr)
\wedge
\bigl(\beta(e_l)e^k-\beta(e_k)e^l\bigr)\nonumber\\
={}&\sum_{i<j}\sum_{k<l}R_{ij,kl}\Bigl[
 \alpha(e_j)\beta(e_l)e^i\wedge e^k
-\alpha(e_j)\beta(e_k)e^i\wedge e^l\nonumber\\
&\hspace{42mm}
-\alpha(e_i)\beta(e_l)e^j\wedge e^k
+\alpha(e_i)\beta(e_k)e^j\wedge e^l
\Bigr].\nonumber
\end{align}
For $\alpha=e^r$ and $\beta=e^s$, pair skew-symmetry gives
\[
\begin{aligned}
 \left(\sum_{a<b}\sum_{c<d}R_{ab,cd}D_{ab}e^r\wedge D_{cd}e^s\right)(e_i,e_j)
 &=R_{ir,js}-R_{jr,is}\\
 &=R_{ij,rs},
\end{aligned}
\]
where the last equality is Bianchi. Extend bilinearly to obtain
\begin{equation}\label{eq:WD2-mixed-is-curvature}
\sum_{i<j}\sum_{k<l}R_{ij,kl}D_{ij}\alpha\wedge D_{kl}\beta
=\mathcal R(\alpha\wedge\beta),
\end{equation}
where $\mathcal R$ is the curvature endomorphism of $\Lambda^2V^*$ with the
normalization fixed above.  Combining
\eqref{eq:Ric-induced-two}, \eqref{eq:WD2-middle} and
\eqref{eq:WD2-mixed-is-curvature} gives
\begin{equation}\label{eq:WD2-result}
{
\mathcal W_2^D(R)=\operatorname{Ric}^{[2]}-2\mathcal R.
}
\end{equation}
Under the double-form/endomorphism correspondence,
$\operatorname{Ric}^{[2]}$ is represented by $g\cdot\operatorname{Ric}$,
whereas $\mathcal R$ is represented by $R$.  Hence our direct Clifford
calculation gives
\begin{equation}\label{eq:WD2-doubleform}
{
\mathcal W_2^D(R)=g\cdot\operatorname{Ric}-2R.
}
\end{equation}

\paragraph{Comparison in degree two.}
The established Bourguignon--Labbi expression specializes at $p=2$ to
\[
\mathcal N_2(R)=g\cdot\operatorname{Ric}-2R.
\]
Comparison with the independently derived \eqref{eq:WD2-doubleform} gives
\[
\mathcal W_2^D(R)=\mathcal N_2(R).
\]

\subsubsection{Arbitrary degree}

On $(V^*)^{\otimes p}$ let $T_{ij}^{(r)}$ denote insertion of $(e^{ij})^\sharp$ in the $r$th slot. Individual slot maps need not preserve alternating tensors. We form their compositions in the full tensor space and restrict the displayed sums to $\Lambda^pV^*$. Equation~\eqref{eq:DA-slot-formula} says
\[
D_{ij}=\left.\left(\sum_{r=1}^pT_{ij}^{(r)}\right)\right|_{\Lambda^pV^*},
\]
and therefore
\begin{equation}\label{eq:DADB-slot-split}
D_{ij}D_{kl}
=\sum_{r=1}^pT_{ij}^{(r)}T_{kl}^{(r)}
+\sum_{r\ne s}T_{ij}^{(r)}T_{kl}^{(s)}.
\end{equation}
The first sum consists of same-slot terms.  Applying the degree-one
contraction \eqref{eq:WD-Ricci-contraction} in each slot gives
\begin{equation}\label{eq:WDp-diagonal}
-\sum_{r=1}^p\sum_{i<j}\sum_{k<l}R_{ij,kl}T_{ij}^{(r)}T_{kl}^{(r)}
=\operatorname{Ric}^{[p]}.
\end{equation}

For the off-diagonal part, fix $r<s$.  The ordered pairs $(r,s)$ and
$(s,r)$ give equal contributions: operators in distinct slots commute and
$R_{ij,kl}=R_{kl,ij}$.  Hence
\[
\sum_{r\ne s}=2\sum_{r<s}
\]
at the level of the curvature coupling.  In the two selected slots the
orthonormal-frame calculation is exactly the four-term calculation
\eqref{eq:WD2-four-expansion}; the remaining $p-2$ slots are spectators.
Thus each unordered pair contributes the curvature insertion determined by
$R$, and the factor $2$ comes solely from the two ordered placements.  We
obtain, entirely from the $D$-channel calculation,
\begin{equation}\label{eq:WDp-slot-result}
{
\mathcal W_p^D(R)=\operatorname{Ric}^{[p]}-2R^{[p]},
}
\end{equation}
where $R^{[p]}$ denotes the sum of the curvature insertions over the
unordered pairs of slots.

The standard double-form/endomorphism correspondence now translates the two
terms already computed into
\[
\operatorname{Ric}^{[p]}
\longleftrightarrow
\frac{g^{p-1}\cdot\operatorname{Ric}}{(p-1)!},
\qquad
R^{[p]}
\longleftrightarrow
\frac{g^{p-2}\cdot R}{(p-2)!}.
\]
For $2\le p\le n$ this gives
\begin{equation}\label{eq:WDp-computed}
{
\mathcal W_p^D(R)
=\frac{g^{p-1}\cdot\operatorname{Ric}}{(p-1)!}
-2\frac{g^{p-2}\cdot R}{(p-2)!}.
}
\end{equation}

\subsubsection{Comparison with the classical formula}

For $2\le p\le n$, the Bourguignon--Labbi Weitzenb\"ock curvature transformation is
\begin{equation}\label{eq:BL-Weitzenbock-comparison}
\mathcal N_p(R)
=\frac{g^{p-1}\cdot\operatorname{Ric}}{(p-1)!}
-2\frac{g^{p-2}\cdot R}{(p-2)!},
\end{equation}
Here $\mathcal N_0(R)=0$ and $\mathcal N_1(R)=\operatorname{Ric}$. Compare \eqref{eq:WDp-computed}.

\begin{theorem}\label{thm:same-factor-Weitzenbock}
For every algebraic curvature tensor $R$ and $0\le p\le n$,
\begin{equation}\label{eq:WD-equals-Np}
{
-\sum_{i<j}\sum_{k<l}\langle\mathcal R e^{ij},e^{kl}\rangle D_{ij}D_{kl}
=\mathcal W_p^D(R)=\mathcal N_p(R).
}
\end{equation}
Thus the classical Weitzenb\"ock curvature transformation is realized as the
same-factor quadratic coupling of the degree-preserving Clifford operators.
\end{theorem}

\begin{remark}\label{rem:same-mixed-D-coupling}
On $\D^{p,q}=\Lambda^pV^*\otimes\Lambda^qV^*$ write
\[
D_{ij}^L=D_{ij}\otimes I,
\qquad
D_{ij}^R=I\otimes D_{ij}.
\]
Then the three natural quadratic placements are
\begin{align*}
D_{ij}^LD_{kl}^L&=(D_{ij}D_{kl})\otimes I,\\
D_{ij}^LD_{kl}^R&=D_{ij}\otimes D_{kl},\\
D_{ij}^RD_{kl}^R&=I\otimes(D_{ij}D_{kl}).
\end{align*}
Theorem~\ref{thm:same-factor-Weitzenbock} and
Theorem~\ref{thm:Gamma-sharp} therefore give
\begin{equation}\label{eq:three-D-placements}
{
\begin{aligned}
-\sum_{i<j}\sum_{k<l}R_{ij,kl}D_{ij}^LD_{kl}^L
&=\mathcal N_p(R)\otimes I,\\
\sum_{i<j}\sum_{k<l}R_{ij,kl}D_{ij}^LD_{kl}^R
&=\Gamma_{1,1}(R)=\frac14R\#,\\
-\sum_{i<j}\sum_{k<l}R_{ij,kl}D_{ij}^RD_{kl}^R
&=I\otimes\mathcal N_q(R).
\end{aligned}}
\end{equation}
Thus same-factor coupling gives Weitzenb\"ock curvature, and mixed-factor coupling gives sharp.
\end{remark}

Put $\mathcal D_{ij}=D_{ij}^L+D_{ij}^R$. We use the curvature term of the averaged Lichn\'erowicz Laplacian:
\[
\begin{aligned}
 \operatorname{Ric}_L&=-\frac12\sum_{i<j}\sum_{k<l}R_{ij,kl}\mathcal D_{ij}\mathcal D_{kl},\\
 \Delta_L&=\nabla^*\nabla+\operatorname{Ric}_L.
\end{aligned}
\]
By \eqref{eq:three-D-placements},
\begin{equation}\label{eq:RicL-sharp}
\begin{aligned}
 \operatorname{Ric}_L(\omega)
 &=\frac12\bigl(\omega\circ\mathcal N_p(R)
       +\mathcal N_q(R)\circ\omega\bigr)
       -\Gamma_{1,1}(R)\omega\\
 &=\frac12\bigl(\omega\circ\mathcal N_p(R)
       +\mathcal N_q(R)\circ\omega\bigr)-\frac14R\#\omega.
\end{aligned}
\end{equation}
The unaveraged diagonal curvature action is $2\operatorname{Ric}_L$. The coefficient in \eqref{eq:RicL-sharp} follows directly by expanding $\mathcal D_{ij}\mathcal D_{kl}$; the same factor $1/2$ multiplies its same-factor and mixed terms.

\subsection{A curvature-dynamical consequence}

Define the algebraic quadratic map
\[
 \mathcal Q(R):=-\frac14R\#R,\qquad R\in\D^{2,2}.
\]
This sign is part of the definition of $\mathcal Q$. The curvature reaction contains $-\mathcal Q(R)$:
\[
 2(R\circ R+R\#_{\mathrm H}R)
 =2R\circ R-\mathcal Q(R).
\]
See \cite{HamiltonFour,BoehmWilking}.

\begin{corollary}
\label{cor:sharp-linearization}
With $g$ fixed, let $R,S\in\D^{2,2}$. Then
\begin{equation}\label{eq:Q-Gamma}
{\quad
\mathcal Q(R)=-\Gamma_{1,1}(R)R.
\quad}
\end{equation}
Moreover,
\begin{equation*}
{\quad
\left.\frac{d}{dt}\right|_{t=0}\mathcal Q(R+tS)=-2\Gamma_{1,1}(R)S.
\quad}
\end{equation*}
Equivalently, the central component $\Gamma_{1,1}(R)$ is, up to the universal
factor $-2$, the linearization operator of the algebraic sharp quadratic map at $R$.
\end{corollary}

\begin{proof}
Equation~\eqref{eq:Q-Gamma} is Theorem~\ref{thm:Gamma-sharp} with
$\omega=R$.  Since $\#$ is symmetric and bilinear,
\[
\begin{aligned}
\left.\frac{d}{dt}\right|_{t=0}\mathcal Q(R+tS)
&=\left.\frac{d}{dt}\right|_{t=0}
\left[-\frac14(R+tS)\#(R+tS)\right]\\
&=-\frac14(R\#S+S\#R)
=-\frac12R\#S.
\end{aligned}
\]
Applying Theorem~\ref{thm:Gamma-sharp} once more gives
$\left.\frac{d}{dt}\right|_{t=0}\mathcal Q(R+tS)=-2\Gamma_{1,1}(R)S$.
\end{proof}

\begin{corollary}
\label{cor:ricci-flow-Gamma}
Let $g'(t)=-2\operatorname{Ric}(g(t))$, and let $R(t)$ be the covariant curvature double form with the convention fixed above. Then
\[
R'(t)=-\nabla^*\nabla R
-R\circ g\operatorname{Ric}
-g\operatorname{Ric}\circ R
+2R\circ R
+\Gamma_{1,1}(R)R.
\]
Thus the Lie-algebraic sharp contribution to the reaction term is the central
homogeneous component of the higher Clifford transform applied to the
curvature tensor itself.
\end{corollary}

\begin{proof}
In a moving orthonormal frame, \cite{BoehmWilking} gives
\[
 \partial_t\widehat R=-\nabla^*\nabla\widehat R
 +2\widehat R\circ\widehat R+2\widehat R\#_{\mathrm H}\widehat R.
\]
Choose $\partial_te_i=\operatorname{Ric}^{\sharp}e_i$. Differentiating the four arguments adds
\[
 R\circ(g\operatorname{Ric})+(g\operatorname{Ric})\circ R
\]
to the covariant derivative $R'$. Subtract these terms and use
\[
 2R\#_{\mathrm H}R=\frac14R\#R=\Gamma_{1,1}(R)R.
\]
\end{proof}

\begin{remark}
For $R=\kappa g^2/2$, choose the orthonormal basis
$e^{ij}=e^i\wedge e^j$, $i<j$.  Then
\[
\Gamma_{1,1}(R)
=\kappa\sum_{i<j}D_{ij}^LD_{ij}^R.
\]
On $e^1\otimes e^2$ the $(1,2)$ term gives
\[
 \Gamma_{1,1}(R)(e^1\otimes e^2)=-\kappa e^2\otimes e^1.
\]
Thus the mixed term in the averaged Lichn\'erowicz action is
$-\Gamma_{1,1}(R)(e^1\otimes e^2)=\kappa e^2\otimes e^1$. Moreover,
\[
 \Gamma_{1,1}(R)g=(n-1)\kappa g,\qquad
 \mathcal N_1(R)=(n-1)\kappa g,\qquad\operatorname{Ric}_L(g)=0.
\]
For the covariant round-sphere solution, the reaction formula gives
\[
 R'=\bigl[-4(n-1)+2+2(n-2)\bigr]\kappa R
 =-2(n-1)\kappa R,
\]
as required by metric scaling.
\end{remark}

\subsection{Consequences of the sharp-product identification}

\begin{corollary}\label{cor:sharp-transpose}
For $R\in\D^{2,2}$ and $\omega\in\D$,
\[
\bigl(R\#\omega\bigr)^t=R^t\#\omega^t.
\]
In particular, if $R^t=R$, then $R\#$ commutes with transposition and therefore
preserves the symmetric and skew-symmetric subspaces of every
$\D^{p,p}$.
\end{corollary}

\begin{proof}
The higher Clifford transform is compatible with transposition:
interchanging the two tensor factors exchanges $D_L(\alpha)$ and
$D_R(\alpha)$.  Thus from \eqref{eq:Gamma11R},
\[
\bigl(\Gamma_{1,1}(R)\omega\bigr)^t
=\Gamma_{1,1}(R^t)(\omega^t).
\]
Multiply by $4$ and use Theorem~\ref{thm:Gamma-sharp}.
\end{proof}

\begin{corollary}\label{cor:sharp-hodge}
Let $R\in\D^{2,2}$.  The central Clifford channel commutes separately with
the factorwise Hodge stars:
\[
{
[*_L,\Gamma_{1,1}(R)]=0,
\qquad
[*_R,\Gamma_{1,1}(R)]=0.
}
\]
Consequently,
\[
[*_D,\Gamma_{1,1}(R)]=0,
\qquad *_D:=*_L*_R.
\]
Equivalently, the sharp action satisfies
\[
*_L(R\#\omega)=R\#(*_L\omega),
\qquad
*_R(R\#\omega)=R\#(*_R\omega),
\]
and hence
\[
*_D(R\#\omega)=R\#(*_D\omega).
\]
\end{corollary}

\begin{proof}
Since $D(\alpha)=-\rho(\alpha)$ and the Hodge star is $\mathrm{SO}(V)$-equivariant,
\[
 [*,D(\alpha)]=0.
\]
For $R=\sum_\lambda\alpha_\lambda\otimes\beta_\lambda$, \eqref{eq:Gamma11R} gives
\[
\begin{aligned}
 [*_L,\Gamma_{1,1}(R)]
 &=\sum_\lambda[*_L,D_L(\alpha_\lambda)]D_R(\beta_\lambda)=0,\\
 [*_R,\Gamma_{1,1}(R)]
 &=\sum_\lambda D_L(\alpha_\lambda)[*_R,D_R(\beta_\lambda)]=0.
\end{aligned}
\]
Now use $*_D=*_L*_R$ and $R\#=4\Gamma_{1,1}(R)$.
\end{proof}

\begin{corollary}
\label{cor:lich-coupling}
For $\omega\in\D^{p,q}$ the Lichnerowicz curvature action may be written
\[
{
\operatorname{Ric}_L(\omega)
=\frac12\bigl(\omega\circ\mathcal N_p(R)+
\mathcal N_q(R)\circ\omega\bigr)
-\Gamma_{1,1}(R)(\omega).
}
\]
Thus the first two terms are the one-factor Weitzenb\"ock contributions and
the final term is precisely the bidegree-preserving two-factor component of
the coupled Clifford transform.
\end{corollary}

\begin{proof}
Substitute $R\#\omega=4\Gamma_{1,1}(R)(\omega)$ from
Theorem~\ref{thm:Gamma-sharp} into \eqref{eq:RicL-sharp}.
\end{proof}

\subsection{Normalization and sign checks}\label{subsec:audit}

We record checks distinguishing exterior, composition and Clifford products.

\paragraph{Clifford square.}
For a unit covector $a$,
\[
C_L(a)^2
=\muL{a}^2-\muL{a}\iL{a}-\iL{a}\muL{a}
 +\iL{a}^2=-\id.
\]
The same holds on the right.  Thus our convention corresponds to the Clifford
relation
\[
uv+vu=-2\langle u,v\rangle.
\]
Transported to the composition algebra this becomes
\[
\chi_a\circ\chi_a=-|a|^2e^g,
\]
which agrees with \eqref{eq:chi-clifford}.

\paragraph{Composition unit.}
In an orthonormal basis,
\[
e^g=\sum_{r=0}^n\sum_{|I|=r}e^I\otimes e^I.
\]
Hence $\mathcal T(e^g)$ is the identity on every homogeneous exterior power.
This verifies directly that the scalar term on the right-hand side of the
composition-algebra Clifford relation must be $e^g$ and not the exterior unit
$1\in\D^{0,0}$.

\paragraph{Metric commutator.}
On $\D^{p,q}$, Proposition~\ref{prop:Lc-sl2} gives
\[
c\mu_g-\mu_gc=(n-p-q)\id.
\]
For $\omega=1\in\D^{0,0}$ this reads $c(g)=n$, which is immediate from
\[
c(g)=\sum_{i,j}\iL{e^i}\iR{e^i}(e^j\otimes e^j)=n.
\]
At the opposite extreme, on $\D^{n,n}$ the sign is reversed, as required by
adjunction.  This verifies the shift $-n$ in the definition of $H_+$.

\paragraph{Bianchi commutator.}
On $\D^{1,0}$ the Bianchi map $\Bop$ vanishes for degree reasons, whereas
\[
\Bop^*\Bop(e^a\otimes1)=0,\qquad
\Bop\Bop^*(e^a\otimes1)=e^a\otimes1.
\]
Thus $[\Bop,\Bop^*]=+\id$ on $\D^{1,0}$, agreeing with
$H_-=N_L-N_R$.  This fixes the order in \eqref{eq:SSstar}; reversing the
commutator would reverse the $\mathfrak{sl}_2$ convention.

\paragraph{The $2$-form Clifford decomposition.}
If $a,b$ are orthonormal and $\alpha=a\wedge b$, then on a scalar
\[
C_L(\alpha)1=a\wedge b=\mu^L_\alpha1,
\]
whereas on the top exterior degree only the contraction term survives.  On a
$1$-form $u$ the degree-zero part is
\[
-D_L(\alpha)u.
\]
These three tests separately fix the signs in
\[
C_L(\alpha)=\mu^L_\alpha-D_L(\alpha)-i^L_\alpha.
\]

\paragraph{Sharp-product factor.}
For a symmetric curvature operator, write
\[
 R=\sum_{i<j}\sum_{k<l}R_{ij,kl}\,e^{ij}\otimes e^{kl},
 \qquad R_{ij,kl}=R_{kl,ij}.
\]
The diagonal degree-zero derivation on double forms is
\[
 \rho_d(e^{ij})=-\bigl(D_{ij}^L+D_{ij}^R\bigr).
\]
The mixed contribution to
$\rho_d(e^{ij})\rho_d(e^{kl})$ is
\[
 D_{ij}^LD_{kl}^R+D_{ij}^RD_{kl}^L.
\]
By pair symmetry and commutation of the two factors, the unaveraged diagonal action has mixed term
\[
 -2\sum_{i<j}\sum_{k<l}R_{ij,kl}D_{ij}^LD_{kl}^R
 =-\frac12R\#.
\]
After averaging, this becomes $-\frac14R\#=-\Gamma_{1,1}(R)$, as in \eqref{eq:RicL-sharp}.

\paragraph{Hodge equivariance.}
Finally, the derivation $\rho(\alpha)$ defined above commutes
with the Hodge star.  Since $D(\alpha)=-\rho(\alpha)$, the same is true for
$D(\alpha)$.  Therefore the central curvature component must commute with both
$*_L$ and $*_R$.  This is consistent with
Corollary~\ref{cor:sharp-hodge} and provides a useful check that no hidden
degree-dependent Hodge sign belongs in the degree-preserving component.

\subsection{The higher extension beyond curvature}

The contracted-wedge formula \eqref{eq:Gamma-Wedge-rs} resolves every channel in arbitrary bidegree. The central $(2,2)$ channel and its same-factor counterpart have the interpretations
\[
 \Gamma_{1,1}(R)=\frac14R\#,\qquad
 \mathcal W_p^D(R)=\mathcal N_p(R).
\]
For higher-degree coefficients, analogous geometric interpretations and uniform formulas for the full array of Hodge symmetries remain questions for further work.

\section{Conclusion}
The regular-action identities \eqref{eq:intro-main-chi} realize the two Clifford actions inside Greub's composition algebra. Their coupled action gives the four metric maps, two commuting $\mathfrak{sl}_2$-triples and the higher channel family \eqref{eq:Gamma-rs-decomp}. In the curvature case, the central channel is sharp, the same-factor coupling is Weitzenb\"ock curvature, and the extreme channels recover the scalar, traceless-Ricci and Weyl summands. These identifications preserve the classical conventions of the cited double-form calculus.

\section*{Acknowledgments}
The author would like to express his sincere gratitude to Professor
Mohammed-Larbi Labbi for introducing him to the fascinating subject of double
forms and for the mathematical inspiration that grew from that introduction.

\end{document}